\documentclass[english,11pt]{amsart}
\usepackage[english]{babel}

\usepackage[matrix,arrow]{xy}
\xyoption{all}
\usepackage{amscd,amssymb,amsfonts,amsmath}

\usepackage{graphics}
\usepackage{epsfig}
\usepackage{mathrsfs}
\DeclareMathAlphabet{\mathscrbf}{OMS}{mdugm}{b}{n}

\newcommand\mypagesizel{
\textwidth= 6.5in
\textheight=9in
\voffset-.55in
\hoffset -0.75in
\marginparwidth=56pt
}

\newcommand{\Pic}{\textup{Pic}}

\newcommand{\p}[0]{{\mathbb P}}

\newcommand{\Frob}[1]{\textup{F}_{\textup{abs},#1}}
\newcommand{\Frobabs}{\textup{F}_{\textup{abs}}}

\newcommand{\Hol}{\textup{Hol}}

\newcommand{\Nef}{\textup{Nef}}

\newcommand{\Psef}{\textup{Psef}}

\newcommand{\NS}{\textup{N}^1}
\newcommand{\N}{\textup{N}}

\newcommand{\NE}{\overline{\textup{NE}}}

\renewcommand{\phi}{\varphi}
\newcommand{\map}{\dashrightarrow}

\newcommand{\into}{\hookrightarrow}
\newcommand{\onto}{\twoheadrightarrow}
\newcommand{\wt}{\widetilde}
\newcommand{\wh}{\widehat}

\renewcommand{\le}{\leqslant}
\renewcommand{\ge}{\geqslant}

\newcommand{\mult}{\textup{mult}}
\newcommand{\rank}{\textup{rank}}

\newcommand{\B}{\textup{B}}

\newcommand{\bH}{\textup{\textbf{H}}}

\newcommand{\bQ}{\mathbb{Q}}

\newcommand{\bS}{\textup{\textbf{S}}}
\newcommand{\bT}{\textup{\textbf{T}}}

\newcommand{\bX}{\textup{\textbf{X}}}

\newcommand{\sA}{\mathscr{A}}

\newcommand{\sE}{\mathscr{E}}

\newcommand{\sG}{\mathscr{G}}
\newcommand{\sH}{\mathscr{H}}

\newcommand{\sL}{\mathscr{L}}

\newcommand{\sN}{\mathscr{N}}
\newcommand{\sO}{\mathscr{O}}

\newcommand{\sbfE}{\mathscrbf{E}}
\newcommand{\sbfG}{\mathscrbf{G}}

\mypagesizel

\newtheorem{thm}{Theorem}[section]
\newtheorem*{thm*}{Theorem}

\newtheorem{lemma}[thm]{Lemma}
\newtheorem{cor}[thm]{Corollary}

\newtheorem{prop}[thm]{Proposition}

\theoremstyle{definition}
\newtheorem{defn}[thm]{Definition}

\newtheorem{say}[thm]{}
\newtheorem{exmp}[thm]{Example}

\newtheorem{defn-thm}[thm]{Definition-Theorem}

\newtheorem{rem}[thm]{Remark}

\theoremstyle{remark}
\newtheorem{notation}[thm]{Notation}

\newtheorem*{not-and-def}{Notation and definitions}
\newtheorem{claim}[thm]{Claim}
\newtheorem{assumption}[thm]{Assumption}

\numberwithin{equation}{section}

\begin{document}

\title[A decomposition theorem for singular spaces with trivial canonical class]{A decomposition theorem for singular spaces with trivial canonical class of dimension at most five}

\author{St\'ephane \textsc{Druel}}

\address{St\'ephane Druel: Institut Fourier, UMR 5582 du CNRS, Universit\'e Grenoble Alpes, CS 40700, 38058 Grenoble cedex 9, France} 

\email{stephane.druel@univ-grenoble-alpes.fr}



\begin{abstract}
In this paper we partly extend the Beauville-Bogomolov decomposition theorem to the singular setting. We show that any complex projective variety of dimension at most five with canonical singularities and numerically trivial canonical class admits a finite cover, \'etale in codimension one, that decomposes as a product of an Abelian variety, and singular analogues of irreducible Calabi-Yau and irreducible holomorphic symplectic
varieties.
\end{abstract}

\maketitle

\tableofcontents

\section{Introduction}

The Beauville-Bogomolov decomposition theorem asserts that any compact K\"ahler manifold with numerically trivial canonical bundle admits an \'etale cover that decomposes into a product of a torus, and irreducible,
simply-connected Calabi-Yau, and holomorphic symplectic manifolds (see \cite{beauville83}). Moreover, the decomposition of the simply-connected part corresponds to a decomposition of the tangent bundle into a direct sum whose summands are integrable and stable with respect to any polarisation.

With the development of the minimal model program, 
it became clear that singularities arise as an inevitable part of higher dimensional life.
If X is any complex projective manifold with Kodaira dimension $\kappa(X) = 0$, standard
conjectures of the minimal model program predict the existence of a birational contraction
$X \map X'$, where $X'$ has terminal singularities and $K_{X'}\equiv 0$.
This makes imperative to extend the Beauville-Bogomolov decomposition theorem to the singular setting.

Building on recent extension theorems for differential forms on singular spaces, Greb, Kebekus and Peternell  
prove an analogous decomposition theorem for the tangent sheaf of projective varieties with canonical singularities and numerically trivial canonical class.

\begin{thm}[{\cite[Theorem 1.3]{gkp_bo_bo}}]\label{thm:infinitesimal_beauville_bogomolov}
Let $X$ be a normal complex projective variety with at worst canonical singularities. Assume that $K_X \equiv 0$. Then there exists an abelian variety $A$
as well as a projective variety $\wt X$ with at worst canonical
singularities, a finite cover $f: A \times \wt X \to X$, \'etale in
codimension one, and a decomposition
$$
T_{\wt X} = \oplus_{i\in I} \sE_i
$$
such that the following holds.
\begin{enumerate}
\item The $\sE_i$ are integrable subsheaves of $T_{\wt X}$, with $\det(\sE_i)\cong\sO_{\wt X}$.
\end{enumerate}
Further, if $g\colon \wh X \to \wt X$ is any finite cover, \'etale in codimension one, then the following properties hold in addition.
\begin{enumerate}
\item[(2)] The sheaves $(g^*\sE_i)^{**}$ are stable with respect to any polarisation on $\wh X$.
\item[(3)] The irregularity $h^1(\wh X,\sO_{\wh X})$ of $\wh X$ is zero.
\end{enumerate}
\end{thm}

Based on Theorem \ref{thm:infinitesimal_beauville_bogomolov} above, they argue 
in \cite[Section 8]{gkp_bo_bo}
that the natural building 
blocks for any structure theory of projective varieties with canonical singularities and numerically trivial canonical class are
canonical varieties with \emph{strongly stable} tangent sheaf (see Definition~\ref{def:strongStab} for this notion). In dimension no more than five, they also show that canonical varieties with strongly
stable tangent sheaf fall into two classes, which naturally generalise the
notions of irreducible Calabi-Yau and irreducible holomorphic-symplectic
manifolds, respectively. 

\medskip

The main result of our paper is the following decomposition theorem.

\begin{thm}\label{thm:main_intro}
Let $X$ be a normal complex projective variety of dimension at most $5$,
with at worst klt singularities. Assume that $K_X \equiv 0$. Then there exists an abelian variety $A$
as well as a projective variety $\wt X$ with at worst canonical
singularities, a finite cover $f: A \times \wt X \to X$, \'etale in
codimension one, and a decomposition
$$\wt X \cong \prod_{i\in I} Y_i \times \prod_{j\in J} Z_j$$
of $\wt X$ into normal projective varieties with trivial canonical class, such that the following holds.

\begin{enumerate}
\item We have $h^0 \big(\wt Y_i, \Omega_{\wt Y_i}^{[q]} \big) = 0$ for all numbers $0 < q < \dim X$ and all finite
covers $\wt Y_i \to Y_i$, \'etale in codimension one.
\item There exists a reflexive $2$-form $\sigma \in
H^0\big(Z_j, \Omega_{Z_j}^{[2]} \big)$ such that $\sigma$ is everywhere
non-degenerate on the smooth locus ${Z_j}_{\textup{reg}}$ of $Z_j$, and such that for all finite covers 
$f\colon \wt Z_j \to Z_j$, \'etale in codimension one, the exterior algebra of global
reflexive forms is generated by $f^{[*]}\sigma \in H^0\big(\wt Z_j, \Omega_{\wt Z_j}^{[2]} \big)$.
\end{enumerate}
\end{thm}

\begin{rem}
The decomposition of $\wt X$ corresponds to the decomposition of 
$T_{\wt X}$ given by Theorem \ref{thm:infinitesimal_beauville_bogomolov} above up to permutation of the summands.
\end{rem}

The proof of the Beauville-Bogomolov decomposition theorem heavily uses K\"ahler-Einstein
metrics and the solution of the Calabi conjecture. But these results are not yet available in the singular setting.
Instead, the proof of Theorem \ref{thm:main_intro} relies on Theorem \ref{thm:infinitesimal_beauville_bogomolov}
and on sufficient criteria to guarantee that a given foliation has algebraic leaves. 
In \cite{bost}, Bost proved an arithmetic algebraicity criterion for leaves of algebraic foliations defined over a number field. Building on his result, we obtain the following algebraicity criterion. 

\begin{thm}\label{theorem:alg_int_flat_case}
Let $X$ be a normal complex projective variety of dimension $n$ with at worst terminal singularities, and
let $H$ be an ample Cartier divisor. Let $$T_X = \oplus_{i\in I} \sG_i \oplus \sE$$ 
be a decomposition of $T_X$ into involutive subsheaves. 
Suppose that for any finite cover $g\colon \wh X \to X$, \'etale in codimension one, the sheaf $(g^*\sG_i)^{**}$ is $g^*H$-stable. Suppose furthermore that 
$c_1(\sG_i)\cdot H^{n-1}=0$ and 
$c_2(\sG_i)\cdot H^{n-2} \neq 0$ for each $i\in I$.
Suppose finally that $\sE$ is $H$-semistable with
$c_1(\sE)\cdot H^{n-1} = c_1(\sE)^2\cdot H^{n-2} = 0$ and $c_2(\sE)\cdot H^{n-2} = 0$.
Then there exists an abelian variety $A$
as well as a projective variety $\wt X$ with at worst terminal
singularities, and a finite cover $f: A \times \wt X \to X$, \'etale in
codimension one, such that $(f^{*}\sE)^{**} = T_{A\times \wt X/\wt X}$
as subsheaves of $T_{A\times \wt X}$.
\end{thm}

\begin{rem}
In the setup of Theorem \ref{theorem:alg_int_flat_case}, there exists a finite cover 
$f\colon  \wt X \to X$ \'etale in codimension one such that 
$(f^{*}E)^{**}$ is a locally free, flat sheaf on $\wt X$ (\cite[Theorem 1.20]{gkp_flat}). In particular, if 
the \'etale fundamental group
$\wh \pi_1(X_{\textup{reg}})$ is finite, then the conclusion of Theorem \ref{theorem:alg_int_flat_case} follows easily from the description of the Albanese map of mildly singular varieties whose canonical divisor is numerically trivial in
\cite[Prop.~8.3]{kawamata85}. On the other hand, \cite[Corollary 3.6]{gkp_bo_bo} reduces the study of varieties with trivial canonical class to those with zero augmented irregularity
(see Definition \ref{defn:augmented_irregularity} for this notion), and it is expected that the \'etale fundamental group of their smooth locus is finite (see \cite[Section 8]{gkp_bo_bo} and 
\cite[Theorem 1.5]{gkp_flat}). This is true if $\dim X \le 4$ by \cite[Corollary 8.25]{gkp_bo_bo}, providing an alternative proof of Theorem \ref{theorem:alg_int_flat_case} in this case.
\end{rem}

The geometric counterpart of Bost's arithmetic algebraicity criterion, independently obtained by Bogomolov and McQuillan (\cite{bogomolov_mcquillan01}), and very recently extended by Campana and P\u{a}un (\cite{campana_paun15}) leads to the following algebraicity criterion.

\begin{thm}\label{theorem:pereira_touzet_conjecture}
Let $X$ be a normal complex projective variety of dimension $n$, and let $H$ be an ample Cartier divisor.
Suppose that $X$ is smooth in codimension two. Let $$T_X = \sE \oplus \sG$$ be a decomposition of $T_X$
into involutive subsheaves, where $\sE$ is $H$-stable, $\det(\sE)\cong \sO_X$ and 
$c_2(\sE)\cdot H^{n-2} \neq 0$. 
Suppose furthermore that $\sE$ has rank at most $3$.
Then $\sE$ has algebraic leaves. 
\end{thm}

\noindent Theorem \ref{theorem:pereira_touzet_conjecture} confirms a conjecture of Pereira and Touzet in some special cases (see \cite[Remark 6.5]{pereira_touzet}).
It is one of the main technical contribution of this paper. 

\medskip

This paper is organized as follows.

In section 3, we review basic definitions and results about foliations on normal varieties. 

In section 4, we 
show that algebraic integrability of direct summands in the decomposition of the tangent bundle given by Theorem \ref{thm:infinitesimal_beauville_bogomolov} 
leads to a decompositon of the variety, perhaps after passing to a finite cover that is \'etale in codimension one (see Theorem \ref{theorem:kawamata_abelian_factor} and Proposition \ref{proposition:alg_int_towards_dec}). 

Section 5 is devoted to the proof of Theorem \ref{theorem:alg_int_flat_case}. 

In sections 6-8, we prove Theorem \ref{theorem:pereira_touzet_conjecture}. In the setup of Theorem \ref{theorem:pereira_touzet_conjecture}, we show that either $\sE$ satisfies the Bost-Campana-P\u{a}un algebraicity criterion in Proposition \ref{proposition:alg_integrability}, or $\sE$ admits a holomorphic Riemannian metric. This is an immediate consequence of our study of stable reflexive sheaves of rank at most $3$ with numerically trivial first Chern class and
pseudo-effective tautological line bundle in Section 6. For a precise statement, see Theorem \ref{theorem:stable_sheaf_psef}. If $\sE$ admits a holomorphic metric, then it follows from Proposition \ref{proposition:holomorphic_connection} that $\sE$ admits a holomorphic connexion, yielding a contradiction.

In section 9, we finally prove Theorem \ref{thm:main_intro}.

\

\noindent {\bf Acknowledgements.} We would like to thank Jorge V. \textsc{Pereira} and
Fr\'ed\'eric \textsc{Touzet} for helpful discussions. The author would also like to thank Cinzia \textsc{Casagrande} for important suggestions. Lemma \ref{lemma:product_versus_contraction} goes back to her.

\section{Notation, conventions, and basic facts}

\begin{say}[Global Convention]
Throughout the paper a variety is a reduced and irreducible scheme of finite type over a field.  
\end{say}

\begin{say}[Differentials, reflexive hull]
Given a variety $X$, we denote the sheaf of K\"{a}hler differentials by
$\Omega^1_X$. 
If $0 \le p \le \dim X$ is any number, write
$\Omega_X^{[p]}:=(\Omega_X^p)^{**}$.
The tangent sheaf will be denoted by $T_X:=(\Omega_X^1)^*$.

\medskip

Given a variety $X$, $m\in \mathbb{N}$, and coherent sheaves $\sE$ and $\sG$ on $X$, write
$\sE^{[m]}:=(\sE^{\otimes m})^{**}$, $S^{[m]}\sE:=(S^m\sE)^{**}$, $\det(\sE):=(\Lambda^{\rank \,\sE}\sE)^{**}$, and $\sE \boxtimes \sG := (\sE \otimes\sG)^{**}.$
Given any morphism $f \colon Y \to X$, write 
$f^{[*]}\sE:=(f^*\sE)^{**}.$ 
\end{say}

\begin{say}[Stability]
The word ``stable'' will always mean ``slope-stable with respect to a
given polarisation''. Ditto for semistability.

\begin{defn}[{\cite[Definition 7.2]{gkp_bo_bo}}]\label{def:strongStab}
Let $X$ be a normal complex projective variety of dimension $n$, and let $\sG$ be a reflexive
coherent sheaf of $\sO_X$-modules. We call $\sG$ \emph{strongly stable}, if
for any finite morphism $f : \wt X \to X$ that is \'etale in codimension
one, and for any choice of ample divisors $\wt H_1, \ldots, \wt H_{n-1}$ on $\wt X$, the reflexive pull-back $f^{[*]} \sG$ is stable with respect to $(\wt H_1, \ldots, \wt H_{n-1})$.
\end{defn}
\end{say}

\begin{say}[Nef and pseudo-effective cones]
Let $X$ be a complex projective variety and consider the finite dimensional dual $\mathbb{R}$-vector spaces 
$$\N_1(X)=\big(\{1-\text{cycles}\}/\equiv\big)\otimes\mathbb{R} \quad \textup{and} \quad \NS(X)=\big(\Pic(X)/\equiv\big)\otimes\mathbb{R},$$
where $\equiv$ denotes numerical equivalence. 
Set $\NS(X)_\mathbb{Q}=\big(\Pic(X)/\equiv\big)\otimes\mathbb{Q}$.
The \emph{Mori cone} of $X$ is the closure $\NE(X)\subset \N_1(X)$ of the cone 
$\textup{NE}(X)$
spanned by classes of effective curves. Its dual cone is the \emph{nef cone} 
$\Nef(X)\subset\NS(X)$, which by Kleiman's criterion is the closure of the cone spanned by ample classes.
The closure of the cone spanned by effective classes  in $\NS(X)$ is the \emph{pseudo-effective cone} $\Psef(X)$. 
\end{say}

\begin{say}[Projective space bundles]
If $\sE$ is a coherent locally free sheaf of $\sO_X$-modules on a complex variety $X$, 
we denote by $\p_X(\sE)$ the variety $\textup{Proj}_X\big(\textup{Sym}(\sE)\big)$,
and by $\sO_{\p_X(\sE)}(1)$ its tautological line bundle. 

\begin{lemma}\label{lemma:tautological_psef}
Let $X$ be a complex projective variety, let $H$ be an ample Cartier divisor, and let $\sE$ be a coherent locally free sheaf of $\sO_X$-modules.
Then $[\sO_{\p_X(\sE)}(1)]\in \N^1\big(\p_X(\sE)\big)$ is pseudo-effective if and only if 
there exists $c>0$ such that
$h^0\big(X,S^{i}\sE\otimes\sO_X(jH)\big)=0$
for any positive integer $j$ and any natural number $i$ satisfying 
$i>cj$.
\end{lemma}

\begin{proof}
Set $Y:=\mathbb{P}_X(\sE)$, denote by $\sO_Y(1)$ the tautological line bundle on $Y$, and by $\pi\colon Y \to X$ the natural morphism. If $[\sO_Y(1)]\in \N^1(Y)$ is not 
pseudo-effective, then $\sO_Y(m)\otimes\pi^*\sO_X(H)$ is not 
pseudo-effective either for a sufficiently large positive integer $m$. 
Let now $i$ and $j$ be positive integers such that $i>mj$. Then $\sO_Y(i)\otimes\pi^*\sO_X(jH)$ is not 
pseudo-effective as well, and hence
$h^0\big(X,S^{i}\sE\otimes\sO_X(jH)\big)=h^0\big(Y,\sO_Y(i)\otimes\pi^*\sO_X(jH)\big)=0$ by the projection formula.

Conversely, suppose that $[\sO_Y(1)]\in N^1(Y)$ is pseudo-effective. Pick $m_0>0$ such that 
$\sO_Y(1)\otimes\pi^*\sO_X(m_0H)$ is ample. Then, for each positive integer $m$, the line bundle
$\sO_Y(m+1)\otimes\pi^*\sO_X(m_0H)$ is big, and hence there exists a positive integer $k$ such that 
$h^0\big(X,S^{km+k}\sE\otimes\sO_X(km_0 H)\big)=h^0\big(Y,\sO_Y(km+k)\otimes\pi^*\sO_X(km_0H)\big) \neq 0$. This completes the proof of the lemma.
\end{proof}

\end{say}

\begin{say}[Chern classes]
We will need to discuss intersection numbers of line bundles with Chern classes of reflexive sheaves on singular varieties. We use \cite[Chapter 3]{fulton} as our main reference for Chern classes on varieties over a field.
Given a variety $X$ over a field, we denote by $\textup{A}_k(X)$ the group of $k$-dimensional cycles modulo rational equivalence.

\begin{defn}
Let $X$ be a variety of dimension $n$, and let $\sG$ be coherent sheaf of $\sO_X$-modules. 
Let $X^\circ \subset X_{\textup{reg}}$ be the maximal open set where $\sG$ is locally free. 
Assume that the complement of $X^\circ$ in $X$ has codimension at least $k+1$ for some positive interger $k$.
The $k$-th Chern class 
$c_k(\sG)$ of $\sG$ is the image of $\textbf{c}_{k}(\sG_{|X^\circ})\cap [X^\circ] \in \textup{A}_{n-k}(X^\circ)$
under the isomorphism
$\textup{A}_{n-k}(X^\circ) \cong \textup{A}_{n-k}(X)$, where 
$\textbf{c}_{k}(\sG_{|X^\circ})\colon \textup{A}_{\bullet}(X^\circ) \to \textup{A}_{\bullet-k}(X^\circ)$ is the Chern class operation.
\end{defn}

\begin{rem}
Varieties with terminal singularities are smooth in codimension two (see \cite[Corollary 5.18]{kollar_mori}). So that first and second Chern classes of coherent sheaves are well-defined on these varieties.
\end{rem}
 
The following result is probably well-known to experts. We include a proof here for the reader's convenience. Notice that much of the intersection theory developed in \cite[Chapters 1-10]{fulton} is valid for schemes, separated and of finite type, over a noetherian regular scheme (see \cite[Chapter 20]{fulton}). 
Given a scheme $X$ of finite type over a noetherian regular scheme $S$, we denote by $\textup{A}_k(X/S)$ the group of 
\emph{relative dimension} $k$ cycles modulo rational equivalence.

\begin{lemma}\label{lemma:conservation_numbers}
Let $T$ be an integral noetherian scheme of dimension $m$, let $X$ be an integral scheme of dimension $n$, and let $X \to T$ be a dominant proper morphism. Let $\sG$ be coherent sheaf of $\sO_X$-modules, and let $H$ be a Cartier divisor on $X$. Let $X^\circ \subset X$ be the maximal open set where $\sG$ is locally free. 
Assume that the complement of $X^\circ$ in $X$ has codimension at least $k+1$ for some positive integer $k$. Then there exists a dense open set $T^\circ \subset T$ such that the intersection number
$c_k(\sG_t)\cdot H_t^{n-m-k}$ is independent of $t\in T^\circ$.
\end{lemma}

\begin{rem}
Let $t\in T$. The scheme $X_t$ is viewed as a scheme over the residue field of $t$. 
If the complement of $X^\circ \cap X_t$ in $X_t$ has codimension at least $k+1$, then $c_k(\sE_t)$ is 
well-defined. 
\end{rem}

\begin{proof}[Proof of Lemma \ref{lemma:conservation_numbers}]
Replacing $T$ by a dense open set, we may assume that $X \to T$ is flat, and that, for any point $t\in T$, the complement of $X^\circ \cap X_t$ in $X_t$ has codimension at least $k+1$. We will show
that the intersection number $c_k(\sG_t)\cdot H_t^{n-m-k}$ is independent of $t\in T$. In order to prove our claim, we may assume without loss of generality that $T=\textup{Spec}\, R$ for some discrete valuation ring $R$.
Let $\eta$ be the generic point of $T$, and let $t$ be its closed point. 

Let now $c_k(\sG)$ be the image of $\textbf{c}_{k}(\sG_{|X^\circ})\cap [X^\circ] \in 
\textup{A}_{n-m-k}(X^\circ/S)$ under the isomorphism $\textup{A}_{n-m-k}(X^\circ/X) 
\cong \textup{A}_{n-m-k}(X/S)$, where 
$\textbf{c}_{k}(\sG_{|X^\circ})\colon \textup{A}_{\bullet}(X^\circ/S) \to \textup{A}_{\bullet-k}(X^\circ/S)$ is the Chern class operation. 
The inclusion $X_\eta \subset X$ induces a pull-back morphism 
$${-}_\eta\colon \textup{A}_{\bullet}(X/S) \to \textup{A}_{\bullet}(X_\eta),$$
and the regular embedding $X_t \subset X$ induces a Gysin homomorphism
$${-}_t\colon\textup{A}_{\bullet}(X/S) \to \textup{A}_{\bullet}(X_t).$$
By \cite[Chapter 20]{fulton}, there is a
specialization map
$$s\colon \textup{A}_{\bullet}(X_\eta) \to \textup{A}_{\bullet}(X_t)$$
such that $s \circ {-}_\eta={-}_t$.

The degree $0$ component of $s$ preserves degrees by \cite[Proposition 20.3 a)]{fulton} and, by \cite[Example 20.3.3]{fulton}, we have 
$s(\alpha_\eta\cdot H) =  s (\alpha_\eta) \cdot H_{|X_t}$
for any cycle $\alpha \in \textup{A}_{\bullet}(X_\eta)$ on the generic fiber. It follows that we have
$${c_k(\sG)}_{\eta}\cdot H_{|X_\eta}^{n-m-k}={c_k(\sG)}_{t}\cdot H_{|X_t}^{n-m-k}.$$
Notice that $[X]_t=[X_t]$ since $X$ is flat over $T$ and $\{t\} \subset T$ is a regular embedding (see \cite[Theorem 6.2 b)]{fulton}).
Using functoriality of Gysin homomorphisms (see \cite[Theorem 6.5]{fulton}) together with \cite[Proposition 6.3]{fulton}, one readly checks that the image of ${c_k(\sG)}_{t}$
under the isomorphism
$\textup{A}_{n-k}(X_t) \onto \textup{A}_{n-k}(X^\circ\cap X_t)$ is 
$c_k(\sG_{|X^\circ\cap X_t})$. This implies that ${c_k(\sG)}_{t}=c_k(\sG_t)\in \textup{A}_{n-k}(X_t)$.
Similarly, using functoriality of flat pull-backs together with
\cite[Theorem 3.2 d)]{fulton}, we see that
${c_k(\sG)}_{\eta}=c_k(\sG_\eta) \in \textup{A}_{n-k}(X_\eta)$. This completes the proof of the lemma.
\end{proof}
\end{say}

\begin{say}[Singularities]
We refer to \cite[Section 2.3]{kollar_mori} for details.
Let $X$ be a normal complex projective variety. 
Suppose that $K_X$ is $\mathbb{Q}$-Cartier, i.e.,  some non-zero multiple of it is a Cartier divisor. 
Let $\beta:\wh X\to X$ be a resolution of singularities of $X$. 
This means that $\wh X$ is a smooth projective
variety, $\beta$ is a birational projective morphism whose exceptional locus is the union of prime divisors $E_i$'s, and the divisor $\sum E_i$ has simple normal crossing 
support.  
There are uniquely defined rational numbers $a(E_i,X)$ such that
$$
K_{\wh X} = \beta^*K_X+\sum a(E_i,X)E_i.
$$
The numbers $a(E_i,X)$ do not depend on the resolution $\beta$,
but only on the valuations associated to the divisors $E_i$. 
We say that $X$ is \emph{terminal} (respectively, \emph{canonical}) if, for some resolution of singularities $\beta:\wh X\to X$ of $X$, $a(E_i,X)> 0$ (respectively, $a(E_i,X)\ge 0)$
for every $\beta$-exceptional prime divisor $E_i$.
If these conditions hold for some log resolution of $X$, then they hold for every  
log resolution of $X$.
\end{say}

\section{Foliations}

We first recall the basic facts concerning foliations.

\begin{defn}
A \emph{foliation} on  a normal  variety $X$ is a coherent subsheaf $\sG\subseteq T_X$ such that
\begin{enumerate}
\item $\sG$ is closed under the Lie bracket, and
\item $\sG$ is saturated in $T_X$. In other words, the quotient $T_X/\sG$ is torsion-free.
\end{enumerate}

The \emph{rank} $r$ of $\sG$ is the generic rank of $\sG$.
The \emph{codimension} of $\sG$ is defined as $q:=\dim X-r$. 

Let $X^\circ \subset X_{\textup{reg}}$ be the maximal open set where $\sG_{|X_{\textup{reg}}}$ is a subbundle of $T_{X_{\textup{reg}}}$. 
A \emph{leaf} of $\sG$ is a connected, locally closed holomorphic submanifold $L \subset X^\circ$ such that
$T_L=\sG_{|L}$. A leaf is called \emph{algebraic} if it is open in its Zariski closure.

The foliation $\sG$ is said to be \emph{algebraically integrable} if its leaves are algebraic.
\end{defn}

\begin{say}[Foliations defined by $q$-forms] \label{q-forms}
Let $\sG$ be a codimension $q$ foliation on an $n$-dimenional normal variety $X$.
The \emph{normal sheaf} of $\sG$ is $\sN:=(T_X/\sG)^{**}$.
The $q$-th wedge product of the inclusion
$\sN^*\into \Omega^{[1]}_X$ gives rise to a non-zero global section 
 $\omega\in H^0\big(X,\Omega^{q}_X\boxtimes \det(\sN)\big)$
 whose zero locus has codimension at least two in $X$. 
Moreover, $\omega$ is \emph{locally decomposable} and \emph{integrable}.
To say that $\omega$ is locally decomposable means that, 
in a neighborhood of a general point of $X$, $\omega$ decomposes as the wedge product of $q$ local $1$-forms 
$\omega=\omega_1\wedge\cdots\wedge\omega_q$.
To say that it is integrable means that for this local decomposition one has 
$d\omega_i\wedge \omega=0$ for every  $i\in\{1,\ldots,q\}$. 
The integrability condition for $\omega$ is equivalent to the condition that $\sG$ 
is closed under the Lie bracket.

Conversely, let $\sL$ be a reflexive sheaf of rank $1$ on $X$, and 
$\omega\in H^0(X,\Omega^{q}_X\boxtimes \sL)$ a global section
whose zero locus has codimension at least two in $X$.
Suppose that $\omega$  is locally decomposable and integrable.
Then  the kernel
of the morphism $T_X \to \Omega^{q-1}_X\boxtimes \sL$ given by the contraction with $\omega$
defines 
a foliation of codimension $q$ on $X$. 
These constructions are inverse of each other. 
\end{say}

\begin{say}[Foliations described as pull-backs] \label{pullback_foliations}

Let $X$ and $Y$ be normal varieties, and let $\varphi\colon X\map Y$ be a dominant rational map that restricts to a morphism $\varphi^\circ\colon X^\circ\to Y^\circ$,
where $X^\circ\subset X$ and  $Y^\circ\subset Y$ are smooth open subsets.

Let $\sG$ be a codimension $q$ foliation on $Y$. Suppose that the restriction $\sG^\circ$ of $\sG$ to $Y^\circ$ is
defined by a twisted $q$-form
$\omega_{Y^\circ}\in H^0\big(Y^\circ,\Omega^{q}_{Y^\circ}\otimes \det(\sN_{\sG^\circ})\big)$.
Then $\omega_{Y^\circ}$ induces a non-zero twisted $q$-form 
$\omega_{X^\circ}\in 
H^0\Big(X^\circ,\Omega^{q}_{X^\circ}\otimes (\varphi^\circ)^*\big(\det(\sN_\sG)_{|Y^\circ}\big)\Big)$, which in turn defines a codimension $q$ foliation $\sE^\circ$ on $X^\circ$. We say that
the saturation $\sE$ of $\sE^\circ$ in $T_X$
\emph{is the pull-back of $\sG$ via $\varphi$},
and write $\sE=\varphi^{-1}\sG$.
\end{say}

\begin{defn}\label{definition:pull-back}
Let $\psi\colon X \to Y$ be an equidimensional dominant morphism of normal varieties, and let $D$ be a Weil $\mathbb{Q}$-divisor on $Y$. The pull-back $\psi^*D$ of $D$ is defined as follows. We define 
$\psi^*D$ to be the unique $\mathbb{Q}$-divisor on $X$ whose restriction to 
$\psi^{-1}(Y_{\textup{reg}})$ is $(\psi_{|\psi^{-1}(Y_{\textup{reg}})})^*D_{|\psi^{-1}(Y_{\textup{reg}})}$. This construction agrees with the usual pull-back if $D$ is $\mathbb{Q}$-Cartier.
\end{defn}

We will use the following notation.

\begin{notation}
Let $\psi \colon X \to Y$ be an equidimensional dominant morphism of normal varieties. 
Write
$K_{X/Y}:=K_X-\psi^*K_Y$. We refer to it as the \emph{relative canonical divisor of $X$ over $Y$}.
\end{notation}

\begin{notation}
Let $\psi \colon X \to Y$ be an equidimensional dominant morphism of normal varieties. 
Set $$R(\psi)=\sum_{D} \big(\psi^*D-{(\psi^*D)}_{\textup{red}}\big)$$
where $D$ runs through all prime divisors on $Y$. We refer to it
as the \emph{ramification divisor of $\psi$}.
\end{notation}

\begin{defn}
Let $\sG$ be a foliation on a normal projective variety $X$.
The \textit{canonical class} $K_{\sG}$ of $\sG$ is any Weil divisor on $X$ such that  $\sO_X(-K_{\sG})\cong \det(\sG)$. 
\end{defn}

\begin{exmp}\label{example:canonical_class_foliation}
Let $\psi \colon X \to Y$ be an equidimensional dominant morphism of normal varieties, and let $\sG$ be the foliation on $X$ induced by $\psi$. A straightforward computation shows that
$$K_\sG=K_{X/Y}-R(\psi).$$
\end{exmp}

\begin{say}[The family of leaves] \label{family_leaves} 
Let $X$ be normal projective variety, and let $\sG$ be an algebraically integrable foliation on $X$.
We describe the \emph{family of leaves} of $\sG$
(see \cite[Remark 3.12]{codim_1_del_pezzo_fols}).

There is a unique normal projective variety $Y$ contained in the normalization 
of the Chow variety of $X$ 
whose general point parametrizes the closure of a general leaf of $\sG$
(viewed as a reduced and irreducible cycle in $X$).
Let $Z \to Y\times X$ denotes the normalization of the universal cycle.
It comes with morphisms

\centerline{
\xymatrix{
Z \ar[r]^{\beta}\ar[d]_{\psi} & X \\
 Y &
}
}
\noindent where $\beta\colon Z\to X$ is birational and, for a general point $y\in Y$, 
$\beta\big(\psi^{-1}(y)\big) \subset X$ is the closure of a leaf of $\sG$.
The variety $Y$ is called the \emph{family of leaves} of $\sG$.

Suppose furthermore that $K_\sG$ is $\bQ$-Cartier. 
There is a canonically defined effective Weil $\bQ$-divisor $B$ on $Z$ such that 
\begin{equation}\label{eq:universal_canonical_bundle_formula}
K_{Z/Y}-R(\psi)+B \sim_\mathbb{Q} \beta^* K_\sG,
\end{equation}
where $R(\psi)$ denotes the ramification divisor of $\psi$. 
\end{say}

\begin{rem}\label{remark:canonical_bundle_formula_B_exceptional}
In the setup of \ref{family_leaves}, 
notice that $B$ is $\beta$-exceptional.
This is an immediate consequence of Example \ref{example:canonical_class_foliation}.
\end{rem}

We will need the following easy observation.

\begin{lemma}\label{lemma:almost_proper}
Let $X$ be a normal complex projective variety with at worst $\mathbb{Q}$-factorial terminal singularities, and let $\sG$ be an algebraically integrable foliation on $X$. Suppose that $K_X$ is pseudo-effective
and that $K_\sG\sim_\mathbb{Q} 0$. 
Let $\psi\colon Z \to Y$ be the family of leaves, and let $\beta\colon Z \to X$ be the natural morphism.
Then $\phi:= \psi\circ \beta^{-1}$ is an almost proper map, and $K_{\beta^{-1}\sG}\sim_\mathbb{Q} 0$.
\end{lemma}

\begin{proof}
Notice that $\sG$ is induced by $\phi:= \psi\circ \beta^{-1} \colon X \map Y$.

It follows from \ref{family_leaves} that 
there is a canonically defined effective Weil $\bQ$-divisor $B$ on $Z$ such that 
\begin{equation}\label{eq:canonical_bundle_formula_1}
K_{\beta^{-1}\sG}+B=K_{Z/Y}-R(\psi)+B \sim_\mathbb{Q} \beta^* K_{\sG} \sim_\mathbb{Q} 0,
\end{equation}
where $R(\psi)$ denotes the ramification divisor of $\psi$. 
Recall from Remark \ref{remark:canonical_bundle_formula_B_exceptional}, that $B$ is $\beta$-exceptional.
Moreover, since $X$ has $\mathbb{Q}$-factorial terminal singularities, there exists an effective $\mathbb{Q}$-divisor $E$ on $X$ such that
\begin{equation}\label{eq:canonical_bundle_formula_2}
K_{Z}=\beta^*K_{X}+E \quad\text{and}\quad\textup{Supp}(E)=\textup{Exc}(\beta).
\end{equation}
From equations \eqref{eq:canonical_bundle_formula_1} and \eqref{eq:canonical_bundle_formula_2}, we obtain
\begin{equation}\label{eq:canonical_bundle_formula_3}
R(\psi)\sim_\mathbb{Q} \beta^*K_{X}-\psi^*K_{Y}+B+E.
\end{equation}
Consider a general fiber $F$ of $\psi$. 
Equation \eqref{eq:canonical_bundle_formula_3} then
shows that $$(\beta^*K_{X}+B+E)_{|F}\sim_\mathbb{Q}0.$$ Since $B$ and $E$ are both effective divisors, 
and since $K_X$ is pseudo-effective,
we must have $E \cap F=\emptyset.$ The equality
$\textup{Exc}(\beta)=\textup{Supp}(E)$ then shows that 
$\phi$ is an almost proper map. 

By the adjunction formula, $K_F\sim_\mathbb{Z}{K_Z}_{|F}$, and thus $K_F$ is pseudo-effective.
Applying \cite[Corollary 4.5]{druel15} to $\psi$, we see that 
$K_{\beta^{-1}\sG}=K_{Z/Y}-R(\psi)$ is pseudo-effective. Equation \eqref{eq:canonical_bundle_formula_1} then shows that 
\begin{equation*}
K_{\beta^{-1}\sG}\sim_\mathbb{Q} 0\quad\text{and}\quad B=0.
\end{equation*}
This finishes the proof of the lemma. 
\end{proof}

It is well-known that an algebraically integrable regular foliation on a complex projective manifold is induced by a morphism onto a normal projective variety (see \cite[Proposition 2.5]{hwang_viehweg}). The next proposition extends this result to some foliations on mildly singular varieties.

\begin{prop}\label{proposition:regular_foliation_morphism}
Let $X$ be a normal complex projective variety with at worst $\mathbb{Q}$-factorial terminal singularities, and
let $T_X = \sE \oplus \sG$ be a decomposition of $T_X$
into involutive subsheaves. Suppose that $K_X$ is pseudo-effective and that $\det(\sE)\cong \sO_X$. Then there exist a finite cover $f\colon \wt X \to X$, \'etale in codimension one,
an open subset $\wt X^\circ\subset {\wt X}_{\textup{reg}}$ with complement of codimension at least two, and 
a projective morphism $\wt \phi^\circ\colon \wt X^\circ \to \wt Y^\circ$ such that the following holds.
\begin{enumerate}
\item The morphism $\wt \phi^\circ$ is a locally trivial fibration for the Euclidean topology.
\item The foliation $f^{-1}\sE$ is induced by $\wt \phi^\circ$.
\end{enumerate}
\end{prop}

\begin{proof}
Let $\psi\colon Z \to Y$ be the family of leaves, and let $\beta\colon Z \to X$ be the natural morphism (see \ref{family_leaves}). 
By \cite[Lemma 4.2]{druel15}, there exists a finite surjective morphism
$\gamma\colon Y_1 \to Y$ with $Y_1$ normal and connected such that the following holds. Let $Z_1$ denotes the normalization of $Y_1 \times_Y Z$. Then the induced morphism $\psi_1\colon Z_1 \to Y_1$ has reduced fibers over codimension one points in $Y_1$. 
Hence, we obtain a commutative diagram as follows,

\centerline{
\xymatrix{
Z_1  \ar[d]_{\psi_1}\ar[rr]^{\alpha,\textup{ finite}} & & Z\ar[d]^{\psi} \ar[rr]^{\beta}&& X\ar@{.>}[dll]^{\phi}\\
Y_1 \ar[rr]_{\gamma,\textup{ finite}}&& Y. &&\\
}
}

\begin{claim}\label{claim:decomposition}
The tangent sheaf $T_{Z_1}$ decomposes as a direct sum 
$$T_{Z_1}=(\beta\circ \alpha)^{-1}\sE\oplus (\beta\circ \alpha)^{-1}\sG.$$ 
\end{claim}

\begin{proof}[Proof of Claim \ref{claim:decomposition}]
Set $q:=\rank\,\sE$, and let $\omega \in H^0\big(X,\Omega_X^{[q]}\big)$ a $q$-form defining $\sG$. 
Then ${\beta^*\omega}_{|Z\setminus \textup{Exc}(\beta)}$ extends across $\textup{Exc}(\beta)$ and gives a 
$q$-form $\beta^*\omega \in H^0\big(Z,\Omega_{Z}^{[q]}\big)$ by \cite[Theorem 1.5]{greb_kebekus_kovacs_peternell10}. 
The $q$-form $\alpha^*(\beta^*\omega)\in H^0\big(Z_1,\Omega_{Z_1}^{[q]}\big)$
defines the 
foliation $(\beta\circ\alpha)^{-1}\sG$, and induces an $\sO_{Z_1}$-linear map 
$(\Lambda^{q}T_{Z_1})^{**} \to \sO_{Z_1}$ 
such that the composed morphism of reflexive sheaves of rank one
$$\sigma\colon\det\big((\beta\circ\alpha)^{-1}\sE\big) \to  (\Lambda^{q}T_{Z_1})^{**} \to \sO_{Z_1}$$
is generically non-zero. 
By Lemma \ref{lemma:almost_proper}, we know that $K_{\beta^{-1}\sE}\sim_\mathbb{Q}0$.
A straightforward computation then shows that 
$$K_{(\beta\circ \alpha)^{-1}\sE}=\alpha^*K_{\beta^{-1}\sE}\sim_\mathbb{Q}0,$$
and hence $\sigma$ must be an isomorphism. This immediately implies that 
$T_{Z_1}=(\beta\circ \alpha)^{-1}\sE\oplus (\beta\circ \alpha)^{-1}\sG,$
proving our claim.
\end{proof}

Let $Z_1^\circ \subset \psi_1^{-1}({Y_1}_{\textup{reg}})$ be the maximal open set where 
$\psi_1^{-1}({Y_1}_{\textup{reg}})$ is smooth. Notice that $Z_1^\circ$ has complement of codimension at least two since $\psi_1$ has reduced fibers over codimension one points in ${Y_1}_{\textup{reg}}$. 
The restriction of the tangent map 
$$T{\psi_1}_{|\psi_1^{-1}({Y_1}_{\textup{reg}})}\colon {T_{Z_1}}_{|\psi_1^{-1}({Y_1}_{\textup{reg}})}\to {\psi_1^{\,*}}_{|\psi_1^{-1}({Y_1}_{\textup{reg}})}T_{{Y_1}_{\textup{reg}}}$$
to $(\beta\circ \alpha)^{-1}\sG_{|\psi_1^{-1}({Y_1}_{\textup{reg}})} \subset {T_{Z_1}}_{|\psi_1^{-1}({Y_1}_{\textup{reg}})}$ 
then induces an isomorphism $(\beta\circ \alpha)^{-1}\sG_{|Z_1^\circ}\cong {\psi_1^{\,*}}_{|Z_1^\circ}T_{{Y_1}_{\textup{reg}}}$, and since 
$(\beta\circ \alpha)^{-1}\sG_{|\psi_1^{-1}({Y_1}_{\textup{reg}})}$ and ${\psi_1^{\,*}}_{|\psi_1^{-1}({Y_1}_{\textup{reg}})}T_{{Y_1}_{\textup{reg}}}$
are both reflexive sheaves, we finally obtain an isomorphism of sheaves of Lie algebras
$$\tau\colon(\beta\circ \alpha)^{-1}\sG_{|\psi_1^{-1}({Y_1}_{\textup{reg}})}\cong {\psi_1^{\,*}}_{|\psi_1^{-1}({Y_1}_{\textup{reg}})}T_{{Y_1}_{\textup{reg}}}.$$
Set $m:=\dim Y=\dim Y_1$. Let $y\in {Y_1}_{\textup{reg}}$, and let $U\ni y$ be an open neighborhood of $y$ 
in ${Y_1}_{\textup{reg}}$
with coordinates $y_1,\ldots,y_m$ on $U$. 
A classical result of complex analysis says that 
there exists a unique local $\mathbb{C}^m$-action on $\psi_1^{-1}(U)$ corresponding to the 
flat connexion $(\beta\circ \alpha)^{-1}\sG_{|\psi_1^{-1}(U)}$ on ${\psi_1}_{|\psi_1^{-1}(U)}$.
The local $\mathbb{C}^m$-action on $\psi_1^{-1}(U)$ is given by a holomorphic map
$\Phi \colon W \to \psi_1^{-1}(U)$, where $W$ is an open neighborhood of the neutral section 
$\{0\}\times \psi_1^{-1}(U)$ in $\mathbb{C}^m\times \psi_1^{-1}(U)$ such that
\begin{enumerate}
\item for all $z\in \psi_1^{-1}(U)$, the subset $\{t \in \mathbb{C}^m \,|\, (t,z) \in W\}$ is connected,
\item setting $t \intercal z:=\Phi(t,z)$, we have $0 \intercal z=z$ for all $z\in \psi_1^{-1}(U)$, if $(t+t',z)\in W$, if $(t',z)\in W$ and $(t,t'\intercal z)\in W$, then $(t+t')\intercal z = t\intercal (t'\intercal z)$ holds.
\end{enumerate}
Moreover, the above local $\mathbb{C}^m$-action on $\psi_1^{-1}(U)$ extends the  
local $\mathbb{C}^m$-action on $U \subset \mathbb{C}^m$ given by 
$y_i(t\intercal y)=t_i+y_i(y)$ for any $t=(t_1,\ldots,t_m)\in \mathbb{C}^m$ and $y\in U$ such that 
$\big(y_1^{-1}(t_1+y_1(y)),\ldots,y_m^{-1}(t_m+y_m(y))\big)\in U$. This immediately implies that 
${\psi_1}_{|\psi_1^{-1}({Y_1}_{\textup{reg}})}$ is a locally trivial fibration for the Euclidean topology. Using Lemma \ref{lemma:almost_proper}, we see that
$$\textup{Exc}(\beta\circ\alpha)=\alpha^{-1}\big(\textup{Exc}(\beta)\big)=\psi_1^{-1}\Big(\psi_1\big(\textup{Exc}(\beta\circ\alpha)\big)\Big)\quad\text{and}\quad
\textup{Exc}(\beta)=\psi^{-1}\Big(\psi\big(\textup{Exc}(\beta)\big)\Big).$$
The proposition then follows easily.
\end{proof}

\section{Towards a decomposition theorem}

The main results of this section assert that algebraic integrability of direct summands in the infinitesimal analogue of the Beauville-Bogomolov decomposition theorem (Theorem \ref{thm:infinitesimal_beauville_bogomolov}) leads to a decompositon of the variety, perhaps after passing to a finite cover that is \'etale in codimension one (see Theorem \ref{theorem:kawamata_abelian_factor} and Proposition \ref{proposition:alg_int_towards_dec}). 

First, we recall structure results for varieties with numerically trivial canonical divisor.
The following invariant is relevant in their investigation (see \cite[Definition 3.1]{gkp_bo_bo}).

\begin{defn}\label{defn:augmented_irregularity}
Let $X$ be a normal projective variety. We denote the irregularity of $X$ by
$q(X) := h^1( X,\, \sO_X )$ and define the \emph{augmented irregularity} as
$$
\wt q(X) := \max \bigl\{q(\wt X) \mid \wt X \to X \ \text{a
finite cover, \'etale in codimension one} \bigr\} \in \mathbb N \cup \{ \infty\}.
$$
\end{defn}

\begin{rem}\label{rem:irregularity_bir_invariant}
By a result of Elkik (\cite{elkik}), canonical singularities are rational. It follows that the irregularity is a birational invariant of complex projective varieties with canonical singularities. 
\end{rem}

\begin{rem}\label{remark:augmented_irregularity_finite}
If $X$ is a projective variety with canonical singularities and numerically
trivial canonical class, \cite[Proposition 8.3]{kawamata85} implies that $q(X) \leq \dim X$. If $\wt X \to X$ is any finite cover, \'etale in
codimension one, then $\wt X$ will likewise have canonical singularities
(see \cite[Proposition 3.16]{kollar97}), and numerically trivial canonical class. In summary,
we see that $\wt q(X) \leq \dim X$. The augmented irregularity of
canonical varieties with numerically trivial canonical class is therefore
finite.
\end{rem}

We will need the following easy observation.

\begin{lemma}\label{lemma:augmented_irregularity_birational_morphism}
Let $X$ and $Y$ be normal complex projective varieties with at worst canonical singularities, and let $\beta \colon Y \to X$ be a birational morphism. Suppose that $K_Y \equiv 0$.
Then $\wt q(X) \ge \wt q(Y)$.
\end{lemma}

\begin{proof}
Notice first that $\wt q(Y)$ is finite by Remark \ref{remark:augmented_irregularity_finite} above.
Let 
$g \colon Y_1 \to Y$ be a finite cover, \'etale in codimension one, such that
$h^1(Y_1,\sO_{Y_1})= \wt q(Y)$. Let $f \colon X_1 \to X$ be the Stein factorization
of the composed map $Y_1 \to Y \to X$. Then $f$ is obviously \'etale in codimension one.
From \cite[Proposition 3.16]{kollar97}, we see that $X_1$ has canonical singularities, and hence 
$h^1(X_1,\sO_{X_1})=h^1(Y_1,\sO_{Y_1})$ by Remark \ref{rem:irregularity_bir_invariant}. This finishes the proof of the lemma.
\end{proof}

The following result often reduces the study of varieties with
trivial canonical class to those with $\wt q(X) = 0$ (see also \cite[Proposition 8.3]{kawamata85}).

\begin{thm}[{\cite[Corollary 3.6]{gkp_bo_bo}}]\label{theorem:kawamata_abelian_factor}
Let $X$ be a normal $n$-dimensional projective variety with at worst canonical
singularities. Assume that $K_X$ is numerically trivial. Then there exist
projective varieties $A$, $\wt X$ and a morphism $f\colon  A \times \wt X\to X$ such that
the following holds.
\begin{enumerate}
\item The variety $A$ is Abelian.
\item The variety $\wt X$ is normal and has at worst canonical singularities.
\item The canonical class of $\wt X$ is trivial, $\omega_{\wt X} \cong \sO_{\wt X}$.
\item The augmented irregularity of $\wt X$ is zero, $\wt q(\wt X) = 0$.
\item The morphism $f$ is finite, surjective and \'etale in codimension one.
\end{enumerate}
\end{thm}

Before we give the proof of Proposition \ref{proposition:alg_int_towards_dec}, we need the following auxiliary results. The author would like to thank Cinzia Casagrande who explained Lemma \ref{lemma:product_versus_contraction} to us.

\begin{lemma}\label{lemma:product_versus_contraction}
Let $X_1$, $X_2$ and $Y$ be complex normal projective varieties. Suppose that there exists a surjective morphism with connected fibers $\beta\colon X_1 \times X_2 \to Y$. Suppose furthermore that 
$q(X_1)=0$. Then $Y$ decomposes as a product 
$Y\cong Y_1\times Y_2,$
and there exist surjective morphisms with connected fibers $\beta_1\colon X_1\to Y_1$ and 
$\beta_2\colon X_2\to Y_2$
such that $\beta = \beta_1 \times \beta_2$.
\end{lemma}

\begin{proof}
Let $H$ be an ample Cartier divisor on $Y$.
Note that $\Pic(X_1 \times X_2) \cong \Pic(X_1)\times\Pic(X_2)$ since $q(X_1)=0$. Thus there exist Cartier divisors $G_1$ and $G_2$ on $X_1$ and $X_2$ respectively such that $\beta^*H \sim_\mathbb{Z} \pi_1^*G_1 + \pi_2^*G_2$, where $\pi_i$ is the projection onto $X_i$. 
Let $\beta_i\colon X_i \to Y_i$ be the morphism corresponding to the semiample divisor $G_i$, so that 
$m_i G_i\sim_\mathbb{Z} \beta_i^*H_i$ for some ample Cartier divisor $H_i$ on $Y_i$ and some positive integer
$m_i$.

Let $C \subset X_1 \times X_2$ be a complete curve contracted by $\beta_1 \times \beta_2 \colon X_1\times X_2 \to Y_1\times Y_2$. Then 
$(\pi_i^*G_i) \cdot C=0$, and hence $\beta^*H \cdot C =0$. This implies that $C$ is contracted by $\beta$, and hence $\beta$ factors through $\beta_1\times \beta_2$ by the rigidity lemma. Thus, there exists a morphism $\gamma \colon Y_1\times Y_2 \to Y$ such that $\beta =\gamma \circ (\beta_1\times \beta_2)$. 
Denote by $p_i\colon Y_1\times Y_2 \to Y_i$ the projection onto $Y_i$.
Hence, we obtain a commutative diagram as follows,

\centerline{
\xymatrix{
X_1\times X_2  \ar[rr]^{\beta_1\times\beta_2}\ar[d]_{\pi_i}\ar@/^2pc/[rrrr]^{\beta} && Y_1\times Y_2 \ar[rr]^{\gamma}\ar[d]^{p_i} && Y \\
X_i \ar[rr]_{\beta_i}&& Y_i. && \\
}
} 
\noindent Then
$$(\beta_1 \times \beta_2)^*(m_2 p_1^*H_1+ m_1 p_2^*H_2) \sim_\mathbb{Z} m_1m_2(\pi_1^*G_1+\pi_2^*G_2)
\sim_\mathbb{Z} m_1m_2\beta^*H = (\beta_1 \times \beta_2)^*(m_1 m_2\gamma^*H).$$
This implies that
$$m_2 p_1^*H_1+ m_1 p_2^*H_2 \sim_\mathbb{Z} m_1 m_2\gamma^*H$$
since 
the map $\beta_1 \times \beta_2$ is surjective with connected fibers.
Since $m_2 p_1^*H_1+ m_1 p_2^*H_2$ is ample, we conclude that $\gamma$ is a finite morphism, and hence an isomorphism
since $\beta$ is surjective with connected fibers. This completes the proof of the lemma.
\end{proof}

\begin{say}[Terminalization]
Let $X$ be a normal complex projective variety with at worst canonical singularities. Recall that 
a \emph{$\mathbb{Q}$-factorial terminalization} of $X$ is a birational crepant morphism 
$\beta \colon \wh X \to X$ where $\wh X$ is a $\mathbb{Q}$-factorial projective variety with terminal singularities.  
The existence of $\beta$ is established in \cite[Corollary 1.4.3]{bchm}.
\end{say}

\begin{prop}\label{proposition:birational_decomposition_versus_biregular_decomposition}
Let $X_1$, $X_2$ and $Y$ be complex projective varieties with canonical singularities such that $K_{X_1}$, $K_{X_2}$ and $K_{Y}$ are nef, and let $\phi\colon X_1 \times X_2 \map Y$
be a birational map. Suppose that 
$q(X_1)=0$. Then $Y$ decomposes as a product 
$Y\cong Y_1\times Y_2,$
and there exist birational maps $\phi_1\colon X_1\map Y_1$ and $\phi_2\colon X_2\map Y_2$
such that $\phi = \phi_1 \times \phi_2$.
\end{prop}

\begin{proof}
Let $\beta_1\colon \wh X_1 \to X_1$, $\beta_2\colon \wh X_2 \to X_2$ and $\gamma\colon \wh Y \to Y$ be 
$\mathbb{Q}$-factorial terminalizations of $X_1$, $X_2$, and $Y$ respectively.
Notice that $K_{\wh X_1\times\wh X_2}$ and $K_{\wh Y}$ are nef.
Set $\wh \phi := \gamma^{-1}\circ \phi \circ (\beta_1 \times \beta_2)\colon \wh X_1\times\wh X_2 \map \wh Y$.
Hence, we obtain a commutative diagram as follows,

\centerline{
\xymatrix{
\wh X_1\times\wh X_2  \ar[d]_{\beta_1\times\beta_2}\ar@{.>}[rr]^{\wh\phi} & & \wh Y\ar[d]^{\gamma} \\
X_1\times X_2 \ar@{.>}[rr]_{\phi}&& Y.\\
}
}

Recall from \cite[Th\'eor\`eme 6.5]{bgs} that the product of complex $\mathbb{Q}$-factorial algebraic varieties is $\mathbb{Q}$-factorial. In particular, $\wh X_1\times\wh X_2$ is $\mathbb{Q}$-factorial.

It follows from Lemma \ref{lemma:product_versus_contraction} applied to $\gamma$ and 
Remark \ref{rem:irregularity_bir_invariant} that it suffices to prove Proposition 
\ref{proposition:birational_decomposition_versus_biregular_decomposition}
for 
$\wh \phi$. 

Now, by \cite[Theorem 1]{kawamata_flops}, $\wh \phi$ decomposes into a sequence of flops,
and therefore, using repeatedly Lemma \ref{lemma:product_versus_contraction}, it suffices to prove Proposition 
\ref{proposition:birational_decomposition_versus_biregular_decomposition}
for a flop. Thus, we may assume that there exists a commutative diagram

\centerline{
\xymatrix{
\wh X_1\times\wh X_2  \ar@{.>}[rr]^{\phi}\ar[rd]_{\alpha} & & \wh Y \ar[ld]^{\alpha^+}\\
 & Z & \\
}
}
\noindent where $\alpha$ and $\alpha^+$ are small elementary birational contractions, $K_{\wh X_1\times\wh X_2}$ is numerically
$\alpha$-trivial, and $K_{\wh Y}$ is numerically $\alpha^+$-trivial.
\noindent By Lemma \ref{lemma:product_versus_contraction} applied to $\alpha$,
$Z$ decomposes as a product $Z\cong Z_1\times Z_2,$
and there exist birational morphisms $\alpha_1\colon \wh X_1\to Z_1$ and $\alpha_2\colon \wh X_2\to Z_2$
such that $\alpha = \alpha_1 \times \alpha_2$. Finally, observe that $\alpha_1$ or $\alpha_2$ is an isomorphism since $\rho(\wh X_1\times\wh X_2/Z)=1$. The claim then follows easily.
\end{proof}

We end the preparation for the proof of Proposition \ref{proposition:alg_int_towards_dec} with the following lemma. It reduces the study of varieties with canonical singularities and trivial canonical class to those with terminal singularities.

\begin{lemma}\label{lemma:reduction_terminalization}
Let $X$ be a normal complex projective variety with at worst canonical singularities, and let 
$\beta \colon \wh X \to X$ be a $\mathbb{Q}$-factorial terminalization of $X$.
Let $T_{X} = \oplus_{i\in I} \sE_i$
be a decomposition of $T_X$ into involutive subsheaves with $\det(\sE_i)\cong\sO_X$. 
Then  there is a decomposition $T_{\wh X} = \oplus_{i\in I} \wh\sE_i$ of $T_{\wh X}$
into involutive subsheaves with $\det(\wh\sE_i)\cong\sO_{\wh X}$ such that 
$\sE_i \cong (\beta_*\wh\sE_i)^{**}$.
\end{lemma}

\begin{proof}Notice that $\omega_{\wh X} \cong \sO_{\wh X}$. Denote by $q_i$ the codimension of $\sE_i$,
and consider 
$\omega_i \in H^0\big(X,\Omega_X^{[q_i]}\big)$ a $q_i$-form defining $\sE_i$.
By \cite[Theorem 1.5]{greb_kebekus_kovacs_peternell10}, $\omega_i$ extends to a 
$q_i$-form $\wh \omega_i \in H^0\big(\wh X,\Omega_{\wh X}^{[q_i]}\big)$. Then $\wh \omega_i$ defines a 
foliation $\wh \sE_i \subseteq T_{\wh X}$ with $\det(\wh \sE_i)\cong \sO_{\wh X}(E_i)$
where $E_i$ is the maximal effective divisor on $\wh X$ such that
$\wh \omega_i \in H^0\big(\wh X,\Omega_{\wh X}^{q_i}\boxtimes\sO_{\wh X}(-E_i)\big)$.
The natural map $\oplus_{i\in I} \wh\sE_i \to T_{\wh X}$ being generically injective, we obtain
$$ \sO_{\wh X}(\sum_{i\in I} E_i+E)\cong\det(T_{\wh X}) \cong \sO_{\wh X}$$
for some effective divisor $E$ on $\wh X$. It follows that $E_i=0$ for every $i\in I$, and that 
$T_{\wh X}$ decomposes as a direct sum
$$T_{\wh X} = \oplus_{i\in I} \wh\sE_i$$
of involutive subsheaves with trivial determinants.
The sheaves $\sE_i$ and $(\beta_*\wh\sE_i)^{**}$ agree outside of the $\beta$-exceptional set, and since both are reflexive, we obtain an isomorphism $\sE_i \cong (\beta_*\wh\sE_i)^{**}$. This finishes the proof of Lemma \ref{lemma:reduction_terminalization}.
\end{proof}

The following result together with Theorem \ref{theorem:kawamata_abelian_factor}
can be seen as a first step towards a decomposition theorem.

\begin{prop}\label{proposition:alg_int_towards_dec}
Let $X$ be a normal complex projective variety with at worst canonical singularities,
and let $$T_X = \oplus_{i\in I} \sE_i $$ be a decomposition of $T_X$ into 
involutive subsheaves. Suppose that $\wt q(X)=0$.
Suppose furthermore that the $\sE_i$ are algebraically integrable with $\det(\sE_i)\cong\sO_X$. 
Then there exist a projective variety $\wt X$ with at worst canonical
singularities, a finite cover $f: \wt X \to X$, \'etale in
codimension one, and a decomposition
$$\wt X \cong \prod_{i\in I} Y_i$$
such that the induced decomposition of $T_{\wt X}$
agree with the decomposition $T_{\wt X} = \oplus_{i\in I} f^{[*]}\sE_i $.
\end{prop}

\begin{proof} 
For the reader's convenience, the proof is subdivided into
a number of relatively independent steps.

To prove Proposition \ref{proposition:alg_int_towards_dec}, it is obviously enough to consider the case where $I=\{1,2\}$. Set $\tau(i)=3-i$ for each $i\in I$.

\medskip

\noindent\textit{Step 1. Reduction to $X$ $\mathbb{Q}$-factorial and terminal.}
Let $\beta \colon Z \to X$ be a $\mathbb{Q}$-factorial terminalization of $X$. By Lemma \ref{lemma:reduction_terminalization}, the tangent sheaf $T_{Z}$ decomposes as a direct sum $T_{Z} = \oplus_{i\in I} \sG_i$ of involutive subsheaves with trivial determinants 
such that $\sE_i \cong (\beta_*\sG_i)^{**}$. Notice that $\wt q(Z)=0$ by Lemma \ref{lemma:augmented_irregularity_birational_morphism}.

Suppose that there exists a finite cover $g \colon \wt Z \to Z$, \'etale in codimension one, such that
$\wt Z$ decomposes as a product
$\wt Z \cong \prod_{i\in I} T_i$
such that the induced decomposition of $T_{\wt Z}$
agree with the decomposition $T_{\wt Z} = \oplus_{i\in I} g^{[*]}\sG_i $.
From the K\"{u}nneth formula (see \cite[Theorem 6.7.8]{ega17}), we see that 
$q(T_i)=0$ for any $i\in I$.
Let $f \colon \wt X \to X$ be the Stein factorization
of the composed map $\wt Z \to Z \to X$. Then $f$ is \'etale in codimension one, and
thus $\wt X$ has canonical singularities by \cite[Proposition 3.16]{kollar97}.
Applying Lemma \ref{lemma:product_versus_contraction} to $\wt Z \to \wt X$, we see 
$\wt X$ decomposes as a product
$\wt X \cong \prod_{i\in I} Y_i$
such that the induced decomposition of $T_{\wt X}$
agree with the decomposition $T_{\wt X} = \oplus_{i\in I} f^{[*]}\sG_i $.
We can therefore assume without loss of generality that the following holds.

\begin{assumption}\label{assumption_terminal}
The variety $X$ has at worst $\mathbb{Q}$-factorial terminal singularities.
\end{assumption}

\noindent\textit{Step 2.} For $i\in I$, let $\psi_i\colon Z_i \to Y_i$ be the family of leaves, and let $\beta_i\colon Z_i \to X$ be the natural morphism (see \ref{family_leaves}). Notice that $\sE_i$ is induced by $\phi_i:= \psi_i\circ \beta_i^{-1} \colon X \map Y_i$. By Lemma \ref{lemma:almost_proper}, 
the rational map $\phi_i$ is almost proper. Moreover, it induces a regular map $X_{\textup{reg}} \to Y_i$ since $\sE_i$ is a regular foliation on $X_{\textup{reg}}$.

Let $F_i$ be a general fiber of $\phi_{\tau(i)}$. Then $F_i$ is a normal projective variety with at worst terminal singularities, and $K_{F_i}\sim_\mathbb{Q} 0$.
Set $F_i^\circ:=F_i \cap X_{\textup{reg}}$, and denote by $\wt{F_i^\circ \times_{Y_i} X_{\textup{reg}}}$ the normalization of $F_i^\circ \times_{Y_i} X_{\textup{reg}}$.
Next, we will prove the following.

\begin{claim}\label{claim:alg_int_structure_1}
The natural map
$\wt{F_i^\circ \times_{Y_i} X_{\textup{reg}}}\to X_{\textup{reg}}$ is finite and \'etale over an open subset 
$X_i^\circ\subset X_{\textup{reg}}$ with complement of codimension at least two.
\end{claim}

\begin{proof}[{Proof of Claim \ref{claim:alg_int_structure_1}}]
By Proposition \ref{proposition:regular_foliation_morphism}, there exists a dense open subset 
$Y_i^\circ \subset {Y_i}_{\textup{reg}}$ such that $X_i^\circ := \phi_i^{-1}(Y_i^\circ)$ has complement of codimension at least two in $X$, and such that ${\phi_i}_{|X_i^\circ} \colon X_i^\circ \to Y_i^\circ$ is a projective morphism with irreducible fibers.
Let $P$ be a prime divisor on $Y_i^\circ$, and 
write ${\phi_i}_{|X_i^\circ}^*P=t\, Q$ for some positive integer $t$.
Set $n:=\dim X$, and $m_i:=\dim Y_i$. Notice that $F_i\cap Q\neq\emptyset$.
Since $\sE_1$ and $\sE_2$ are regular foliations at a general point $x$ in $Q$ and $T_{X}=\sE_1\oplus\sE_2$, there exist local analytic coordinates centered at $x$ and $y:=\phi_i(x)$ respectively
such that $\phi_i$ is given
by $(x_1,x_2,\ldots,x_n)\mapsto (x_1^{t},x_2\ldots,x_{m_i})$, and such that 
$F_i$ is 
given by equation $x_{m_i+1}=\cdots=x_n=0$.
The claim then follows from a straightforward local computation.
\end{proof}

\noindent\textit{Step 3. End of proof.}
Let $X_1$ denotes the normalization of $X$ in the function field of 
$\wt{F_1^\circ \times_{Y_1} X_{\textup{reg}}}$. It comes with a finite morphism $f_1 \colon X_1 \to X$
which is \'etale in codimension one by Claim \ref{claim:alg_int_structure_1}.
Let $\wt\phi_1\colon \wt X_1 \map F_1$ be the almost proper rational map induced by $\phi_1$, and let 
$G_1 \subset \wt X_1$ be the Zariski closure of the rational section of $\wt \phi_1$ given by $F_1^\circ \to Y_1$. Finally, let $\wh\phi_2$ denotes the composed map $\wt X_1 \to X \map Y_2$, and set 
$y_2:\wh\phi_2(G_1)=\phi_2(F_1)$. Notice that $\wh \phi_2$ is an almost proper map. 

\begin{claim}\label{claim:holonomy}
The following holds.
\begin{enumerate}
\item The variety $G_1$ is a fiber of the Stein factorization 
$\wt\phi_2 \colon \wt X_1 \map \wt Y_2$
of $\wh\phi_2$.
\item The fiber $\wh\phi_2^{-1}(y_2)=f_1^{-1}(F_1)$ is reduced along $G_1$.
\end{enumerate}
\end{claim}

\begin{proof}[Proof of Claim \ref{claim:holonomy}]
Applying \cite[Proposition 3.16]{kollar97}, we see that
$\wt X$ has terminal singularities. In particular, $X$ is Cohen-Macaulay, and hence so is 
$\wh\phi_2^{-1}(y_2)$ by \cite[Proposition 18.13]{eisenbud}.
By the Nagata-Zariski purity theorem, $f_1$ branches only over the singular set of $X$. This immediately implies that $f_1^{-1}(F_1)$ is smooth in codimension one. Then (2) follows easily.
By Hartshorne's connectedness theorem (see 
\cite[Theorem 18.12]{eisenbud}), we see that irreducible components of $\wh\phi_2^{-1}(y_2)$ are disjoint, proving (1).
\end{proof}

Let $\wt F_2$ be a general fiber of $\wt \phi_1$. Then $\wt F_2$ intersects $G_1$ transversely in a point. 
From Claim \ref{claim:holonomy} (1), it follows that $\wt F_2$ intersects 
a general fiber $\wt F_1$ of $\wt\phi_2$
transversely in a point. This immediately implies that the map $\wt\phi_1\times\wt\phi_2 \colon \wt X_1 \map F_1 \times \wt Y_2$ is birational. But it also implies that $F_1$ and $\wt Y_2$ are birationally equivalent to 
$\wt F_1$ and 
$\wt F_2$ respectively, and we conclude that 
there exists a birational map $\wt X_1 \map \wt F_1\times \wt F_2$.
From the K\"{u}nneth formula (see \cite[Theorem 6.7.8]{ega17}) together with Remark \ref{rem:irregularity_bir_invariant}, we see that 
$q(\wt F_i)=0$ for any $i\in \{1,2\}$.
The conclusion then follows from Proposition \ref{proposition:birational_decomposition_versus_biregular_decomposition}.
This finishes the proof of Proposition \ref{proposition:alg_int_towards_dec}.
\end{proof}

\section{Algebraicity of leaves, I}

In this section we prove Theorem \ref{theorem:alg_int_flat_case}. The proof relies on an algebraicity criterion for leaves of algebraic foliations proved in \cite[Theorem 2.1]{bost}, which we recall now.

\begin{say}
Let $X$ be an algebraic variety over some field $k$ of positive 
characteristic $p$, and let $\sG \subset T_X$ be a subsheaf. We will denote by $\Frobabs \colon X \to X$ the absolute Frobenius morphism of $X$.

The sheaf of derivations $\textup{Der}_{k}(\sO_X)\cong T_X$ is endowed with the $p$-th power operation, which maps any local $k$-derivation $D$ of $\sO_X$ to its $p$-th iterate $D^{[p]}$. 
When $\sG$ is involutive, the map $\Frobabs^*\sG \to T_X/\sG$ which sends 
$D$ to the class in $T_X/\sG$ of $D^{[p]}$ is $\sO_X$-linear. The sheaf $\sG$ is said to be 
\emph{closed under $p$-th powers} if the map $\Frobabs^*\sG \to T_X/\sG$ vanishes.

A connected complex manifold $M$ satisfies the \emph{Liouville property}
when every plurisubharmonic function on $M$ bounded from above is constant (see \cite[Section 2.1.2]{bost}).
Examples of complex manifolds satisfying the Liouville property are provided by affine spaces $\mathbb{C}^n$.

\end{say}

We will use the following notation.

\begin{notation}
If $K$ is a number field, its ring of integers will be denoted by $\sO_K$. For any non-zero prime ideal $\mathfrak{p}$ of $\sO_K$, we let $k(\mathfrak{p})$ be the finite
field $\sO_K/\mathfrak{p}$. We denote by $k(\bar{\mathfrak{p}})$ an algebraic closure of $k(\mathfrak{p})$.
Given a scheme $X$ over $S :=\textup{Spec} \,\sO_K$, we let 
$X_K:=X\otimes K$,
$X_{\mathfrak{p}}:= X \otimes k(\mathfrak{p})$, and $X_{\bar{\mathfrak{p}}}:= X \otimes k(\bar{\mathfrak{p}})$.
Given a sheaf $\sG$ on $X$, we let $\sG_K:=\sG\otimes K$,
$\sG_{\mathfrak{p}}:= \sG \otimes k(\mathfrak{p})$, and $\sG_{\bar{\mathfrak{p}}}:= \sG \otimes k(\bar{\mathfrak{p}})$.

\end{notation}

\begin{thm}[{\cite[Theorem 2.1]{bost}}]\label{thm:bost}
Let X be a smooth algebraic variety over a number field $K$, let $\sG$ be an involutive
subbundle of the tangent bundle $T_X$ of $X$ $($defined over $K)$, and let $x$ be a point in $X(K)$. For some
sufficiently divisible integer $N$, let $\bX$ $($resp. $\sbfG)$ be a smooth model of $X$ over 
$\bS :=\textup{Spec} \,\sO_K[1/N]$
$($resp. a sub-vector bundle of the relative tangent bundle $T_{\bX/\bS}$ such that $\sbfG_K$ coincides
with $\sG)$. Assume that the following two conditions are satisfied.
\begin{enumerate}
\item For almost every non-zero prime ideal $\mathfrak{p}$ of $\sO_K[1/N]$, the subbundle 
$\sbfG_{\mathfrak{p}}$ of $T_{\bX_{\mathfrak{p}}}$
is stable by $p$-th power, where $p$ denotes the characteristic of $k(\mathfrak{p})$.
\item There exists an embedding $\sigma \colon K \into \mathbb{C}$ such that the analytic leaf through 
$x_{\sigma}$ of the involutive holomorphic bundle $\sG_{\sigma}$ on the complex analytic manifold 
$X_{\sigma}(\mathbb{C})$ satisfies the Liouville property.
\end{enumerate}
Then the leaf of $\sG$ through $x$ is algebraic.
\end{thm}

It is well-known that in positive characteristic, there exist semistable vector bundles such that their 
pull-back under the absolute Frobenius morphism is no longer semistable. The next result says that this phenomenon does not occur on projective varieties whose tangent bundle is semistable with zero slope.
It partly extends \cite[Theorem 2.1]{mehta_ramanathan} to the setting where polarisations are given by  
big semiample divisors.
The proof of Proposition \ref{proposition:strong_stability_positive characteristic} is similar to that of \cite[Theorem 2.1]{mehta_ramanathan}.

\begin{prop}\label{proposition:strong_stability_positive characteristic}
Let $X$ be a smooth projective variety over an algebraically closed field $k$ of positive 
characteristic $p$, and let $H$ be a big semiample divisor on 
$X$. Suppose that $T_X$ is $H$-semistable and that $\mu_H(T_X) \ge 0$.
Let $\sE$ be a coherent 
locally free sheaf on $X$. Suppose furthermore that $p \ge \rank \,\sE + \dim X$.
If $\sE$ is $H$-semistable, then so is $\Frobabs^*\sE$.
\end{prop}

\begin{proof}
Suppose that $\Frobabs^*\sE$ 
is not $H$-semistable, and let 
$$\{0\}=\sE_0 \subsetneq \sE_1 \subsetneq \cdots \subsetneq \sE_r:=\Frobabs^*\sE$$ be the Harder-Narasimhan filtration of $\Frobabs^*\sE$.
By \cite[Proposition $1^p$]{shepherd-barron_ss_reduction} (see also \cite[Corollary 2.4]{langer_ss_sheaves}), the canonical connection
on $\Frobabs^*\sE$ induces a non-zero $\sO_X$-linear map
$$\sE_{r-1} \to \sE_r/\sE_{r-1}\otimes \Omega_X^1.$$
Let $C$ be a smooth complete intersection curve of elements of $|mH|$ for 
some sufficiently large integer $m$. By \cite[Corollary 5.4]{langer_ss_sheaves}, the sheaves 
${\sE_i/\sE_{i-1}}_{|C}$ and ${\Omega_X^1}_{|C}$
are semistable. This implies that the sheaves
$(\sE_i/\sE_{i-1})_{|C} \otimes {\Omega_X^1}_{|C}$
are semistable as well by \cite[Remark 3.4]{IMP} using the assumption that $p \ge \rank \,\sE + \dim X$. It follows that
$$\mu_H^{\max}\big(\sE_r/\sE_{r-1}\otimes \Omega_X^1\big) \le 
\mu_H^{\max}\big(\sE_r/\sE_{r-1})$$
using the assumption that $\mu_H(\Omega_X^1) \le 0$.
The inequality $$\mu_H^{\min}\big(\sE_{r-1}\big)=\mu_H\big(\sE_{r-1}/\sE_{r-2}\big) > \mu_H\big(\sE_r/\sE_{r-1})=\mu_H^{\max}\big(\sE_r/\sE_{r-1})$$
then shows that the map 
$\sE_{r-1} \to \sE_r/\sE_{r-1}\otimes \Omega_X^1$
must vanish, yielding a contradiction.
\end{proof}

\begin{rem}
The condition ``$T_X$ $H$-semistable with $\mu_H(T_X) \ge 0$" in Proposition \ref{proposition:strong_stability_positive characteristic} can be weakened to ``$\mu_H^{\textup{min}}(T_X)\ge 0$",  
but we will not need this stronger statement.
\end{rem}

\begin{rem}
We will use Proposition \ref{proposition:strong_stability_positive characteristic} together with 
\cite[Proposition 5.1]{langerAIF} to conclude that, in the setup of Proposition \ref{proposition:strong_stability_positive characteristic}, a $H$-semistable vector bundle $\sE$ is numerically flat if and only if $c_1(\sE)\cdot H^{n-1}=c_1(\sE)^2\cdot H^{n-2}=c_2(\sE)\cdot H^{n-2}=0$, where $n:=\dim X$.
\end{rem}

We end the preparation for the proof of Theorem \ref{theorem:alg_int_flat_case} with the following lemma.

\begin{lemma}\label{lemma:canonical_resolution}
Let $X$ be a normal complex projective variety with at worst canonical singularities, and 
let $T_X = \sE \oplus \sG$ be a decomposition of $T_X$
into sheaves with trivial determinants. 
Suppose 
that $\sE$ is locally free. There exists a resolution of singularities $\beta \colon \wh X \to X$ such that
the following holds.
\begin{enumerate}
\item The morphism $\beta$ induces an isomorphism over the smooth locus $X_{\textup{reg}}$ of $X$.
\item The tangent sheaf  $T_{\wh X}$ decomposes
as a direct sum $T_{\wh X}\cong \beta^*\sE \oplus \wh\sG$. 
\item If $X$, $\sE$, and $\sG$ are defined over a subfield $k\subseteq \mathbb{C}$, then $\wh X$, $\beta$, and $\wh\sG$ are defined over $k$ as well.
\end{enumerate}
\end{lemma}

\begin{proof}
Let $\beta \colon \wh X \to X$ be a resolution of singularities of $X$ such that $\beta_*T_{\wh X}\cong T_X$, and such that $\beta$ induces an isomorphism over $X_{\textup{reg}}$.
The existence of $\beta$ is established in \cite[Corollary 4.7]{greb_kebekus_kovacs10}. It relies on the existence of functorial resolutions of singularities (see \cite[Theorem 3.36]{kollar07}).
Consider the generically injective morphism of locally free sheaves 
$$\beta^*\sE \to \beta^*T_X \cong \beta^*\big(\beta_* T_{\wh X}\big) \to T_{\wh X},$$ where
$\beta^*\big(\beta_* T_{\wh X}\big) \to T_{\wh X}$ is the evaluation map. 
By \cite[Theorem 1.5]{greb_kebekus_kovacs_peternell10}, the projection morphism $T_{X_{\textup{reg}}} \to 
\sE_{|X_{\textup{reg}}}$
extends to a morphism $$T_{\wh X} \to \beta^*\sE.$$ The composed morphism $\beta^*\sE \to T_{\wh X} \to \beta^*\sE$ must be the identity map, and thus $T_{\wh X}$ decomposes
as a direct sum $T_{\wh X}\cong \beta^*\sE \oplus \wh\sG$, where $\wh\sG$ is the kernel of the map 
$T_{\wh X} \to \beta^*\sE$.

Suppose that $X$, $\sE$, and $\sG$ are defined over a subfield $k\subseteq \mathbb{C}$. Then $\wh X$ and $\beta$
are defined over $k$ as well by \cite[Theorem 3.36]{kollar07}. This implies that $\wh \sG$ is also defined over $k$, completing the proof of the lemma.
\end{proof}

Before proving Theorem \ref{theorem:alg_int_flat_case} below, we note the following immediate corollary.

\begin{cor}\label{corollary:augmented_irregularity_versus_second_chern_class}
Let $X$ be a normal complex projective variety of dimension $n$ with at worst terminal singularities, and 
let $T_X = \oplus_{i\in I} \sG_i \oplus \sE$ 
be a decomposition of $T_X$ into involutive subsheaves with trivial determinants. 
Suppose that $\sG_i$ is stable with respect to some ample Cartier divisor $H$, and that 
$c_2(\sG_i) \cdot H ^{n-2} \neq 0$.
Suppose furthermore that $\sE$ is $H$-semistable, and that $c_2(\sE)\cdot H ^{n-2}=0$.
Then $\wt q(X)=\rank \, \sE$.
\end{cor}

\begin{proof}[Proof of Theorem \ref{theorem:alg_int_flat_case}]We maintain notation and assumptions of Theorem \ref{theorem:alg_int_flat_case}.
For the reader's convenience, the proof is subdivided into
a number of relatively independent steps. 

\medskip

\noindent\textit{Step 1. Reduction step.} Set $\sG:=\oplus_{i \in I} \sG_i$. By \cite[Theorem 1.20]{gkp_flat}, there exists a finite cover 
$f_1\colon  X_1 \to X$ that is \'etale in codimension one such that 
$f_1^{[*]}\sE$ is a locally free, flat sheaf on $X_1$. From \cite[Proposition 3.16]{kollar97}, we see that $X_1$ has terminal singularities.
The sheaves $T_{X_1}$ and $f_1^{[*]} \sG \oplus f_1^{[*]}\sE$ agree on $f_1^{-1}(X_{\textup{reg}})$, and since both are reflexive, we obtain a decomposition $T_{X_1} \cong f_1^{[*]} \sG \oplus f_1^{[*]}\sE$
of $T_{X_1}$
into involutive subsheaves with trivial determinants. The same argument also shows that
$f_1^{[*]}\sG\cong \oplus_{i \in I} f_1^{[*]}\sG_i$.
Notice that $f_1$ branches only over the singular set of $X$.
It follows that $c_1(f_1^{[*]}\sG_i)\cdot (f_1^*H)^{n-1}=0$ and that
$c_2(f_1^{[*]}\sG_i)\cdot (f_1^*H)^{n-2} \neq 0$ for each $i\in I$.
Furthermore, $f_1^{[*]}\sE$ is $f_1^*H$-semistable by \cite[Lemma 3.2.2]{HuyLehn},
$f_1^{[*]}c_1(\sE)\cdot (f_1^*H)^{n-1} = c_1(f_1^{[*]}\sE)^2\cdot (f_1^*H)^{n-2} = 0$, and $c_2(f_1^{[*]}\sE)\cdot (f_1^*H)^{n-2} = 0$.
Therefore, we may assume without loss of generality that
the following holds.

\begin{assumption} 
The sheaf $\sE$ is a locally free, flat sheaf.
\end{assumption}

Suppose moreover that $H$ is very ample.

By Lemma \ref{lemma:canonical_resolution}, 
there exists a resolution of singularities $\beta \colon \wh X \to X$ such that $T_{\wh X}$ decomposes
as $T_{\wh X}\cong \beta^*\sE \oplus \wh \sG$. 
Moreover, we may assume that $\beta$ induces an isomorphism over $X_{\textup{reg}}$. 
Set $\wh \sE := \beta^* \sE$, and $\wh H := \beta^* H$. Notice that $\wh \sE$ and $\beta^{[*]}T_X$
are semistable with respect to $\wh H$. Moreover, $T_{\wh X}$ is also $\wh H$-semistable since $T_{\wh X}$ and $\beta^{[*]}T_X$ agree away from the $\beta$-exceptional set.

\medskip

\noindent\textit{Step 2. Algebraicity of leaves over number fields.}
Suppose that $X$, $H$, $\sE$, and the sheaves $\sG_i$ are defined over a number field $K$.
We will show that $\sE$ has algebraic leaves.
Recall from Lemma \ref{lemma:canonical_resolution} that $\wh X$, $\beta$, and $\wh\sG$ are also defined over $K$.

For some sufficiently divisible integer $N$, let $\bX$  
be a flat projective model of $X$ over $\bT:=\textup{Spec}\,\sO_K[1/N]$ with normal (see \cite[Th\'eor\`eme IV.12.2.4]{ega28}) and regular in codimension 2 geometric fibers. 
Let $\sbfE$ (resp. $\sbfG_i$) be a locally free (resp. a coherent) subsheaf
of $T_{\bX/\bT}$ such that $\sbfE_\mathbb{C}$ (resp. 
${\sbfG_i}_\mathbb{C}$) coincides with $\sE$ (resp. $\sG_i$). Suppose moreover that
the sheaves $\sbfG_i$ are flat over $\bT$, 
and that 
$T_{\bX/\bT} = \oplus_{i \in I} \sbfG_i \oplus \sbfE$.
Let $\bH$ be an ample Cartier divisor on $\bX$ such that 
$\bH_\mathbb{C} \sim H$.

Since semistability and geometric stability with respect to an ample divisor are open conditions in flat families of sheaves (see for instance \cite[Proposition 2.3.1]{HuyLehn}), we may assume that the sheaves ${\sbfG_i}_{\bar{\mathfrak{p}}}$ 
are stable with respect to $\bH_{\bar{\mathfrak{p}}}$, and that $\sbfE_{\bar{\mathfrak{p}}}$ 
is $\bH_{\bar{\mathfrak{p}}}$-semistable
for every non-zero prime ideal $\mathfrak{p}$
of $\sO_K[1/N]$. Suppose furthermore that $\sbfE_{\bar{\mathfrak{p}}}$ is involutive.
By Lemma \ref{lemma:conservation_numbers}, we may aslo assume without loss of generality that the following holds: 

\begin{enumerate}
\item $c_1({\sbfE}_{\bar{\mathfrak{p}}})\cdot \bH_{\bar{\mathfrak{p}}}^{n-1} = c_1({\sbfE}_{\bar{\mathfrak{p}}})^2\cdot \bH_{\bar{\mathfrak{p}}}^{n-2} = 0$, and
$c_2({\sbfE}_{\bar{\mathfrak{p}}})\cdot \bH_{\bar{\mathfrak{p}}}^{n-2} = 0$;
\item $c_1({\sbfG_i}_{\bar{\mathfrak{p}}})\cdot \bH_{\bar{\mathfrak{p}}}^{n-1} = 0$, and
$c_2({\sbfG_i}_{\bar{\mathfrak{p}}})\cdot \bH_{\bar{\mathfrak{p}}}^{n-2} \neq 0$.
\end{enumerate}

Let $\wh \bX$ be a smooth projective model of $\wh X$ over $\bT$, 
and let $\boldsymbol{\beta}\colon \wh \bX \to \bX$ be a projective birational morphism such that $\boldsymbol{\beta}_\mathbb{C}$ coincides with $\beta$. Suppose moreover that $\boldsymbol{\beta}_{\bar{\mathfrak{p}}}\colon \wh \bX_{\bar{\mathfrak{p}}} \to \bX_{\bar{\mathfrak{p}}}$ is birational, and that the $\boldsymbol{\beta}_{\bar{\mathfrak{p}}}$-exceptional set maps to a closed subset of codimension at least three in $\bX_{\bar{\mathfrak{p}}}$.
Set $\wh \sbfE := \boldsymbol{\beta}^*\sbfE$. 
Notice that $\wh \sbfE_{\bar{\mathfrak{p}}}$ and $\boldsymbol{\beta}_{\bar{\mathfrak{p}}}^{[*]}T_{\bX_{\bar{\mathfrak{p}}}}$
are semistable with respect to $\wh \bH_{\bar{\mathfrak{p}}}$. 
Moreover, $T_{\wh \bX_{\bar{\mathfrak{p}}}}$ is also $\wh \bH_{\bar{\mathfrak{p}}}$-semistable since $T_{\wh \bX_{\bar{\mathfrak{p}}}}$ 
and $\boldsymbol{\beta}_{\bar{\mathfrak{p}}}^{[*]}T_{\bX_{\bar{\mathfrak{p}}}}$
agree away from the  $\boldsymbol{\beta}_{\bar{\mathfrak{p}}}$-exceptional set. 

Let $\mathfrak{p}$ be non-zero prime ideal
of $\sO_K[1/N]$, and denote by $p$ the characteristic $k(\mathfrak{p})$.
We claim the following.

\begin{claim}
The involutive sub-vector bundle $\wh \sbfE_{\bar{\mathfrak{p}}}$ of 
$T_{\wh \bX_{\bar{\mathfrak{p}}}}$
is closed under $p$-th power. 
\end{claim}

\begin{proof}
We argue by contradiction and assume that $\wh \sbfE_{\bar{\mathfrak{p}}}$ is not closed under $p$-th power, and so neither is $\sbfE_{\bar{\mathfrak{p}}}$.
Thus, the map
$\Frob{\bar{\mathfrak{p}}}^* \sbfE_{\bar{\mathfrak{p}}} \to 
T_{ \bX_{\bar{\mathfrak{p}}}}/\sbfE_{\bar{\mathfrak{p}}}$
induced by the $p$-th power operation
does not vanish identically. 
Since $\sbfE_{\bar{\mathfrak{p}}}$ is semistable with 
$c_1(\sbfE_{\bar{\mathfrak{p}}})\cdot \bH_{\bar{\mathfrak{p}}}^{n-1} = 0$,
and since the sheaves ${\sbfG_i}_{\bar{\mathfrak{p}}}$ are stable with 
$c_1({\sbfG_i}_{\bar{\mathfrak{p}}})\cdot \bH_{\bar{\mathfrak{p}}}^{n-1} = 0$
as well, there exists $i_0\in I$ such that the induced morphism $\Frob{\bar{\mathfrak{p}}}^*\sbfE_{\bar{\mathfrak{p}}} \to {\sbfG_{i_0}}_{\bar{\mathfrak{p}}}$ is surjective in codimension one.

Let $\bS_{\bar{\mathfrak{p}}}$ be a smooth two dimensional complete intersection of general elements of $|m\wh \bH_{\bar{\mathfrak{p}}}|=\boldsymbol{\beta}_{\bar{\mathfrak{p}}}^*|m\bH_{\bar{\mathfrak{p}}}|$ 
for some positive integer $m$.
Since the $\boldsymbol{\beta}_{\bar{\mathfrak{p}}}$-exceptional set maps to a closed subset of codimension at least three in $\bX_{\bar{\mathfrak{p}}}$, 
$\bS_{\bar{\mathfrak{p}}}$ is contained in $\wh \bX_{\bar{\mathfrak{p}}}\setminus \textup{Exc}(\boldsymbol{\beta}_{\bar{\mathfrak{p}}}) \cong \bX_{\bar{\mathfrak{p}}}\setminus (\boldsymbol{\beta}_{\bar{\mathfrak{p}}}\big(\textup{Exc}(\boldsymbol{\beta}_{\bar{\mathfrak{p}}})\big)$, and thus
the restriction of $\wh \bH_{\bar{\mathfrak{p}}}$ to $\bS_{\bar{\mathfrak{p}}}$ is (very) ample. 

By Proposition \ref{proposition:strong_stability_positive characteristic} above, the locally free sheaves $(\Frob{\bar{\mathfrak{p}}}^{\,\circ k})^* \wh\sbfE_{\bar{\mathfrak{p}}}$
are semistable with respect to $\wh \bH_{\bar{\mathfrak{p}}}$. 
Applying \cite[Corollary 5.4]{langer_ss_sheaves} to the locally free sheaves 
$(\Frob{\bar{\mathfrak{p}}}^{\,\circ k})^* \wh\sbfE_{\bar{\mathfrak{p}}}$, we see that 
there exists a positive integer $m$ (that does not depend on $k \ge 1$) such that the restrictions ${(\Frob{\bar{\mathfrak{p}}}^{\,\circ k})^* \wh\sbfE_{\bar{\mathfrak{p}}}}_{|\bS_{\bar{\mathfrak{p}}}}\cong 
{(\Frob{\bar{\mathfrak{p}}}^{\,\circ k})^* \sbfE_{\bar{\mathfrak{p}}}}_{|\bS_{\bar{\mathfrak{p}}}}$ 
are semistable with respect to
${\bH_{\bar{\mathfrak{p}}}}_{|\bS_{\bar{\mathfrak{p}}}}$ for any positive integer $k$. From \cite[Proposition 5.1]{langerAIF}, we conclude that 
${{\sbfE}_{\bar{\mathfrak{p}}}}_{|\bS_{\bar{\mathfrak{p}}}}$ is nef
using the fact that 
$c_1({\sbfE}_{\bar{\mathfrak{p}}})\cdot \bH_{\bar{\mathfrak{p}}}^{n-1} = 0$, 
$c_1({\sbfE}_{\bar{\mathfrak{p}}})^2\cdot \bH_{\bar{\mathfrak{p}}}^{n-2} = 0$,
and
$c_2({\sbfE}_{\bar{\mathfrak{p}}})\cdot \bH_{\bar{\mathfrak{p}}}^{n-2} = 0$.
We view $\bS_{\bar{\mathfrak{p}}}$ as a surface contained in $\bX_{\bar{\mathfrak{p}}}$. 
Observe that  
${\sbfG_{i_0}}_{\bar{\mathfrak{p}}}$ is locally free along $\bS_{\bar{\mathfrak{p}}}$, and that the restriction of
$\Frob{\bar{\mathfrak{p}}}^*\sbfE_{\bar{\mathfrak{p}}} \to {\sbfG_{i_0}}_{\bar{\mathfrak{p}}}$ 
to $\bS_{\bar{\mathfrak{p}}}$ is surjective in codimension one by choice of $i_0$.
Since ${\sbfE_{\bar{\mathfrak{p}}}}_{|\bS_{\bar{\mathfrak{p}}}}$ is nef, we obtain that 
${{\sbfG_{i_0}}_{\bar{\mathfrak{p}}}}_{|\bS_{\bar{\mathfrak{p}}}}$ if nef.
This implies that 
$c_2(\sG_{i_0}) \cdot H^{n-2} = c_2({\sbfG_{i_0}}_{\bar{\mathfrak{p}}})\cdot \bH_{\bar{\mathfrak{p}}}^{n-2} = 0$ by
\cite[Proposition 5.1]{langerAIF} and Lemma \ref{lemma:conservation_numbers}, yielding a contraction. This completes the proof of our claim.
\end{proof}

By \cite[Theorem B]{pereira_touzet},
the leaves of $\wh \sE$ are uniformized by $\mathbb{C}^{\,\rank \,\wh\sE}$, and thus they satisfy the Liouville property.
By Theorem \ref{thm:bost}, we conclude that $\wh \sE$ and hence $\sE$ have algebraic leaves.

\medskip

\noindent\textit{Step 3. Algebraicity of leaves over $\mathbb{C}$.}
To show Theorem \ref{theorem:alg_int_flat_case} in the general case, let $R$ be a subring of $\mathbb{C}$,
finitely generated over $\mathbb{Q}$, and let $\bX$ be a flat projective model of $X$
over $\bT:=\textup{Spec}\, R$ with normal geometric fibers. We may also assume that 
the geometric fibers of $\bX \times_{\bT} {\bT_\mathbb{C}}$ over 
$\bT_\mathbb{C}=\textup{Spec}\, R\otimes\mathbb{C}$
have terminal singularities. 
Let $\sbfE$ (resp. $\sbfG_i$) be a locally free (resp. a coherent) subsheaf
of $T_{\bX/\bT}$ such that $\sbfE_\mathbb{C}$ (resp. 
${\sbfG_i}_\mathbb{C}$) coincides with $\sE$ (resp. $\sG_i$). Suppose moreover that
the sheaves $\sbfG_i$ are flat over $\bT$, 
and that $T_{\bX/\bT} = \oplus_{i\in I} \sbfG_i \oplus \sbfE$.
Let $\bH$ be an ample Cartier divisor on $\bX$ such that 
$\bH_\mathbb{C} \sim H$. As above, we may assume that the sheaves ${\sbfG_i}_{\bar t}$ are stable with respect to $\bH_{\bar t}$, and that $\sbfE_{\bar t}$ is semistable, for any geometric point $\bar t$ of $\bT$.
Suppose furthermore that $\sbfE_{\bar t}$ and the sheaves ${\sbfG_i}_{\bar t}$ are involutive.
By Lemma \ref{lemma:conservation_numbers} again, we may also assume without loss of generality that 
$c_1({\sbfE}_{\bar t})\cdot \bH_{\bar t}^{n-1} = 0$,
$c_1({\sbfE}_{\bar t})^2\cdot \bH_{\bar t}^{n-2} =  0$,
$c_2({\sbfE}_{\bar t})\cdot \bH_{\bar t}^{n-2} =  0$,
$c_1({\sbfG_i}_{\bar t})\cdot \bH_{\bar t}^{n-1} =  0$, and that
$c_2({\sbfG_i}_{\bar t})\cdot \bH_{\bar t}^{n-2}  \neq 0$.
When $t \in \bT$ is a closed point, its residue field is an algebraic number field, and hence $\sbfE_{\bar t}$ has algebraic leaves by the previous step. Applying Lemma \ref{lemma:algebraic_integrability_deformation} below, we conclude that $\sE$ has algebraic leaves.

\medskip

\noindent\textit{Step 4. End of proof.} By \cite[Proposition 2.5]{hwang_viehweg}, $\wh \sE$ is induced by a morphism $\wh \phi \colon \wh X \to \wh Y$ onto a normal projective variety.
As a classical consequence of Yau's theorem on the existence of K\"ahler-Einstein metrics, the geometric generic fiber of $\wh \phi$ is covered by an abelian variety. Thus, there exists a Zariski open set $\wh Y^\circ \subset \wh Y$, and a finite \'etale cover $\wt Y^\circ \to \wh Y^\circ$ such that 
$\wt \phi^\circ\colon \wt X^\circ : = \wh X \times_{\wh Y^\circ} \wt Y^\circ \to \wt Y^\circ$ is a smooth morphism with fibers isomorphic to abelian varieties. Moreover, $\wt  \phi^\circ$ admits a flat connection induced by $\wh \sG_{|{\wh \phi}^{-1}(\wh Y^\circ)}$. Applying Lemma \ref{lemma:isotrivial}, we see that the flat locally free sheaf $\wh \sE$ has finite monodromy, and hence $\sE$ has finite monodromy as well. 
Replacing $X$ by a further cover that is \'etale in codimension one, if necessary, we may assume that 
$\sE \cong \sO_{X}^{\oplus \rank \, \sE}$. 
By Theorem \ref{theorem:kawamata_abelian_factor}, there exists an abelian variety $A$
as well as a projective variety $\wt X$ with at worst canonical
singularities, and a finite cover $f\colon : A \times \wt X \to X$, \'etale in
codimension one, such that the augemented irregularity of $\wt X$ is zero.
The decomposition $T_{A \times \wt X} \cong  \oplus_{i \in I} f^{[*]}\sG_i\oplus f^{[*]}\sE$ of 
$T_{A \times \wt X}$ together with the assumption that $f^{[*]}\sG_i$ is $f^*H$-stable then imply easily that
$f^{[*]}\sE = T_{A\times \wt X/\wt X}$. This finishes the proof of the theorem.
\end{proof}

\begin{lemma}\label{lemma:algebraic_integrability_deformation}
Let $X$ and $T$ be normal noetherian schemes, and let $X \to T$ be a projective morphism with normal geometric fibers. Let $\sE \subseteq T_X$ be a foliation on $X$.
Suppose that $\sE \subseteq T_{X/T}$, and that for every closed point $t\in T$, the induced foliation
on $X_{\bar t}$ has algebraic leaves. Denote by $\eta$ the generic point of $T$. Then the leaves of the foliation
$\sE_{\bar{\eta}}$ on $X_{\bar{\eta}}$ are algebraic.
\end{lemma}

\begin{proof}

The same argument used in the proof of \cite[Lemma 3.2]{fano_fols} shows that there exists 
a normal scheme $B \to T$, projective over $T$, and 
a dominant rational map $X \map B$ over $T$ whose general fibers are leaves of $\sE$.
One only needs to replace the use of the Chow variety $\textup{Chow}(X)$
of $X$ with $\textup{Chow}(X/T)$ the Chow variety of $X$ over $T$.
The proof of \cite[Lemma 3.2]{fano_fols} also shows that  
the foliation
$\sE_{\bar{\eta}}$ on $X_{\bar{\eta}}$ is induced by the rational map 
$X_{\bar{\eta}} \map B_{\bar{\eta}}$, proving the lemma.
\end{proof}

The following result is probably well-known to experts. We include a full proof here for the reader's convenience.

\begin{lemma}\label{lemma:isotrivial}
Let $\phi\colon X \to Y$ be a smooth projective morphism of quasi-projective complex manifolds with fibers isomorphic to abelian varieties. Suppose that $\phi$ is a locally trivial fibration for the Euclidean topology. Then there exists an abelian variety as well as a finite \'etale cover $\wt Y \to Y$ such that $X \times_Y \wt Y  \cong A \times \wt Y$ as varieties over $\wt Y$.
\end{lemma}

\begin{proof}Let $y\in Y$, and denote by $X_y$ the fiber $\phi^{-1}(y)$.
Let $\textup{Aut}^\circ(X_y)\cong X_y$ denotes the neutral component of 
the automorphism group $\textup{Aut}(X_y)$ of $X_y$. 
Recall from \cite[Expos\'e VI$_{\textup{B}}$, Th\'eor\`eme 3.10]{sga3} 
that the algebraic groups
$\textup{Aut}^\circ(X_y)$ fit together to form an abelian scheme $\sA$
over $Y$. Since $\sA$ is locally trivial, there exist 
an abelian variety $A$, and a finite \'etale
cover $Y_1 \to Y$ such that $\sA \times_Y Y_1 \cong A \times Y_1$
as group shemes over $Y_1$. 
This follows from the fact that there is a fine moduli scheme for polarized abelian varieties of dimension $g$, with level $N$ structure and polarization of degree $d$ provided that $N$ is large enough.
In particular, $A$ acts faithfully on $X_1:=X\times_Y Y_1$. 
By \cite[Theorem 2]{brion_action}, there exist a finite \'etale cover 
$\wt X$
of $X_1$ 
equipped with a faithful action of $A$, and an $A$-isomorphism
$\wt X \cong A \times \wt Y$ for some quasi-projective manifold $\wt Y$, where $A$ acts 
trivially on $\wt Y$ and
diagonally on $A \times \wt Y$. 
One readily checks that
the natural morphism
$\wt Y \cong\{0_A\}\times \wt Y \to Y_1$ is \'etale, and that $\wt X \cong X_1\times_{Y_1}\wt Y$ as varieties over $\wt Y$.
The lemma then follows easily.
\end{proof}

\section{Stable reflexive sheaves with pseudo-effective tautological line bundle}

In this section we provide a technical tool for the proof of Theorem \ref{theorem:pereira_touzet_conjecture}:
we study stable reflexive sheaves with numerically trivial first Chern class and
pseudo-effective tautological line bundle.

In \cite{nakayama04}, Nakayama study semistable vector bundle $\sE$ of rank two on complex projective manifold with $c_1(\sE)\equiv 0$ and pseudo-effective tautological class (see \cite[Theorem IV.4.8]{nakayama04} for a precise statement). Our strategy of proof for Theorem \ref{theorem:stable_sheaf_psef} partly follows his line of reasoning.

\begin{thm}\label{theorem:stable_sheaf_psef}
Let $X$ be a normal complex projective variety of dimension $n$, let $H$ be an ample Cartier divisor, and let $\sE$ be a reflexive sheaf of rank $r\in\{1,2,3\}$ on $X$. Suppose that $X$ is smooth in codimension two and that $\sE$ is $H$-stable with $c_1(\sE)\cdot H^{n-1}=0$. 
Suppose furthermore that, for any finite morphism $f\colon \wt X \to X$ that is \'etale in codimension one, the 
reflexive pull-back $f^{[*]}\sE$ is stable with respect to $f^*H$.
Then one of the following holds:
\begin{enumerate}
\item either there exists $c>0$ such that
$h^0\big(X,S^{[i]}\sE\otimes\sO_X(jH)\big)=0$
for any positive integer $j$ and any natural number $i$ satisfying 
$i>cj$,
\item or $c_1(\sE)^2\cdot H^{n-2}=c_2(\sE)\cdot H^{n-2}=0$,
\item or $r=3$, and there exists a finite morphism $f\colon \wt X \to X$ that is \'etale in codimension one,  
and a rank $1$ reflexive sheaf $\sL$ on $\wt X$ with $c_1(\sL)\cdot (f^*H)^{n-1}=0$
such that 
$h^0\big(\wt X,(S^{2}f^{*}\sE)\boxtimes\sL\big)\neq 0$.
\end{enumerate}
\end{thm}

\begin{rem}
The condition on the dimension of the singular locus of $X$ posed in Theorem \ref{theorem:stable_sheaf_psef}
allows to define the Chern class $c_2(\sE)$.
\end{rem}

\begin{rem}\label{remark:pseudo_effectivity_versus_vanishing?}
The condition (1) in the statement of
Theorem \ref{theorem:stable_sheaf_psef} is a way of saying that the tautological line bundle is pseudo-effective on singular spaces (see \ref{lemma:tautological_psef}). 
\end{rem}

\begin{rem}\label{remark:flatness}
In the setup of Theorem \ref{theorem:stable_sheaf_psef}, suppose furthermore that $X$ is smooth, and that $\sE$ is locally free. If $c_1(\sE)^2\cdot H^{n-2}=c_2(\sE)\cdot H^{n-2}=0$, then $\sE$ is flat by a result of Uhlenbeck and Yau (see \cite{uhlenbeck_yau}). In particular, the tautological line bundle is nef.
\end{rem}

The following consequence of Theorem \ref{theorem:stable_sheaf_psef} improves \cite[Theorem 7.7]{bdpp}. The conclusion also holds for $\textup{K}3$-surfaces
by \cite[Theorem IV.4.15 ]{nakayama04}.

\begin{cor}
Let $X$ be a Calabi-Yau complex projective manifold of dimension $3$. Then the tautological line bundle on $\mathbb{P}_X(T_X)$ is not pseudo-effective. 
\end{cor}

\begin{proof}
It is well-known that the tangent sheaf $T_X$ is stable with respect to any polarization $H$, and that $c_2(X)\cdot H^2\neq 0$. We argue by contradiction and assume that the tautological line bundle on $\mathbb{P}_X(T_X)$ is pseudo-effective.
By Theorem \ref{theorem:stable_sheaf_psef}, there exists a line bundle $\sL$ such that $h^0\big(X,(S^2T_X)\otimes\sL\big)=0$. This implies that $\Omega_X^1\cong T_X\otimes\sL$. Taking determinants, we obtain
$\sL^{\otimes 2}\cong\sO_X$. Since $X$ is simply connected, we must have $\sL\cong\sO_X$. On the other hand,
by \cite[Corollary 8]{kobayashi78}, we have $h^0(X,S^2T_X)=0$, yielding a contradiction.
\end{proof}

We collect several examples which illustrate to what extend our result is sharp.

\begin{exmp}
Let $E$ be a smooth complete curve of genus $1$, let $S$ be a projective $\textup{K}3$-surface with Picard number $1$, and
set $X:=E \times S$. Let $\sL$ be a non-torsion line bundle on $E$ of degree $0$, and consider $\sE:=\pi_1^*\sL\oplus \pi_2^*T_X$, where $\pi_1$ and $\pi_2$ are the projections onto $E$ and $S$, respectively. Then $\sE$ is polystable with respect to any polarization, and it is obviously not stable. One readily checks that the conclusion of Theorem \ref{theorem:stable_sheaf_psef} does not hold for $\sE$. 
Set $Y:=\mathbb{P}_{X}(\sE)$. The tautological class $\xi:=[\sO_Y(1)]\in \N^1(Y)$ is pseudo-effective, and not nef in codimension one.
\end{exmp}

\begin{exmp}
Let $X$ be a projective $\textup{K}3$-surface. The tautological line bundle on $\mathbb{P}_X(T_X)$ is not pseudo-effective by \cite[Theorem IV.4.15 ]{nakayama04}. Thus $\sE$ satisfies (1) in the statement of Theorem \ref{theorem:stable_sheaf_psef} by Lemma \ref{lemma:tautological_psef}.
\end{exmp}

\begin{exmp}
Let $X$ be a complex projective manifold, and let $\sL\in\textup{Pic}^0(X)$. Then $\sL$ obviously satisfies (2) in the statement of Theorem \ref{theorem:stable_sheaf_psef}.
\end{exmp}

\begin{exmp}
Let $C$ be a complete curve of genus $g \ge 2$. 
We construct a rank two vector bundle $\sE$ on $C$ of degree $0$ such that, for any \'etale cover $f\colon \wt C \to C$, the pull-back $f^*\sE$ is stable. 
The vector bundle $\sE$ satisfies (2) in the statement of Theorem \ref{theorem:stable_sheaf_psef}.

The construction is very similar to that of Hartshorne in \cite[Theorem I.10.5]{hartshorne_ample_sub}, and so we leave some easy details to the reader.
Pick $c \in C$. By a result of Narasimhan and Seshadri (\cite{narasimhan_seshadri65}), we must construct a
unitary representation $\rho\colon\pi_1(C,c)\to\mathbb{U}(2)$ such that, for any normal subgroup
$H \lhd \pi_1(C,c)$ of finite index, the induced representation 
$H \to \mathbb{U}(2)$ is irreducible.

It is well-known that $\pi_1(C,c)$ is generated by elements 
$a_1,b_1,\ldots,a_g,b_g$ satisfying the relation $$[a_1,b_1]\cdot \cdots \cdot [a_g,b_g]=1.$$
If we have chosen any two unitary matrices
$A_1,B_1 \in \mathbb{U}(2)$, then we can find further unitary matrices
$A_2,B_2,\ldots,A_g,B_g$ satisfying the relation above.
Let $A_1 =
\begin{pmatrix}
   \lambda_1 & 0 \\
   0 & \lambda_2 
\end{pmatrix}
$
where $|\lambda_i|=1$, and $\lambda_1\lambda_2^{-1}$ is not a root of unity. Let $B_1$ be a very general unitary matrix. Then all the entries of all the matrices $B_1^m$ ($m\ge 1$) are non-zero. 
Let $H \lhd \pi_1(C,c)$ be a normal subgroup of finite index. Let $m$ be a positive integer such that 
$A_1^m,B_1^m\in H$. The only invariant subspaces of $A_1^m$ are the subspaces spanned by some subset of the standard basis. This implies that the representation $H \to \mathbb{U}(2)$ is irreducible.
Indeed, in order for $B_1^m$ to have as fixed subspace a subspace generated by a subset of the standard basis, it would have to have some entries zero.
\end{exmp}

\begin{exmp}
Let $X$ be a projective $\textup{K}3$-surface, and consider $\sE:=S^2 T_X$. A straightforward computation shows that $c_2(\sE)=4 \cdot 24$. It is known that $\sE$ is stable with respect to any polarization.
On the other hand, $$S^2(S^2T_X) \cong S^4T_X \oplus \det(T_X)^{\otimes 2} \cong S^4T_X \oplus \sO_X,$$
and hence $h^0(X,S^2\sE)\neq 0$. Thus $\sE$ satisfies (3) in the statement of Theorem \ref{theorem:stable_sheaf_psef} above.
\end{exmp}

We have divided the proof of Theorem \ref{theorem:stable_sheaf_psef} into a sequence of steps, each formulated
as a separate result. Some of these statements might indeed be of independent
interest. The proof of Theorem \ref{theorem:stable_sheaf_psef} then follows quickly from these preliminary steps.

\begin{say}[Divisorial Zariski decomposition]
We briefly recall the definition of the divisorial Zariski decomposition from \cite{nakayama04}.
Let $X$ be a complex projective manifold, let $B$ be a big $\mathbb{R}$-divisor on $X$, and let $P$ be a prime divisor. The asymptotic order of vanishing of $B$ along $P$ is
$$\sigma_P(B)=\inf_G \big\{\mult_P(G)\big\},$$
where the infimum is over all effective $\mathbb{R}$-divisor $G$ with $G\sim_{\mathbb{R}} D$.

Let now $D$ be a pseudo-effective $\mathbb{R}$-divisor, and let $A$ be an ample $\mathbb{R}$-divisor on $X$.
Let 
$$\sigma_P(D) = \lim\limits_{\substack{\varepsilon \to 0 \\ \varepsilon > 0}} \sigma_P(D+\varepsilon A).$$
Then $\sigma_P(D)$ exists and is independent of the choice of $A$. 
There are only finitely many prime divisors $P$ such that $\sigma_P(D)>0$, and the $\mathbb{R}$-divisor 
$N(D):=\sum_P \sigma_P(D) P$ is determined by the numerical equivalence class of $D$. Set 
$P(D):=D-N(D)$.
\end{say}

\begin{say}[Restricted base locus] Let $D$ be a $\mathbb{Q}$-divisor on a smooth projective manifold $X$. 
Let $k$ be a positive integer such that $kD$ is integral. The \emph{stable base locus} of $D$ is
$\textup{B}(D):=\cap_{m\ge 1}\textup{Bs}(mkD)_{\textup{red}}$. It is independent of the choice of $k$.

The \emph{restricted base locus} of an $\mathbb{R}$-divisor $D$ is 
$\textup{B}_{-}(D)=\bigcup_A \textup{B}(D+A)$
where the union is taken over all ample $\mathbb{R}$-divisors $A$ 
such that $D+A$ is a $\mathbb{Q}$-divisor (see \cite[Definition 1.12]{ELMNP_IF}).
It is not known whether the restricted base locus of a divisor is Zariski closed in general, but
it is a countable union of Zariski closed subsets of $X$
by \cite[Proposition 1.19]{ELMNP_IF}. Notice that 
$\textup{B}_{-}(D)\subsetneq X$ if and only if $D$ is pseudo-effective, and that 
$\textup{B}_{-}(D) = \emptyset$ if and only if $D$ nef. Given a positive integer $m$ we say that 
$D$ is \emph{nef in codimension $m$} if each irreducible component of $\textup{B}_{-}(D)$ has codimension $\ge m+1$. In particular, we see that $D$ is nef in codimension one if and only if $N(D)=0$.
\end{say}

The following result is certainly well-known to experts. We include a proof for
lack of references.

\begin{lemma}\label{lemma:multiplicity_non_nef_locus}
Let $X$ be a complex projective manifold, and let $D$ be a pseudo-effective $\mathbb{R}$-divisor on $X$.
Let $B$ be an irreducible component $\textup{B}_{-}(D)$.
Let $\beta_1\colon Y_1 \to X$ be an embedded resolution of $B$, and let $\beta_2\colon Y \to Y_1$ be the 
blow-up of $Y_1$ along the strict transform $B_1$ of $B$ in $Y_1$, with exceptional divisor $E$. 
Set $\beta := \beta_1 \circ \beta_2$. Then $\sigma_E(\beta^*D)>0$.
\end{lemma}

\begin{proof}
Let $A$ be an ample $\mathbb{R}$-divisor on $X$.
By \cite[Lemma 3.3]{ELMNP_IF} applied to the divisorial valuation 
$\mult_E$ of the function field of $X$ given by the order of vanishing at the generic point of $B$,
$$\sigma_E\big(\beta^*(D+A)\big)=\inf_G \big\{\mult_E(\beta^*G)\big\}$$
where the infimum is over all effective $\mathbb{R}$-divisor $G$ with $G\equiv D+A$. 
If $D+A$ is a $\mathbb{Q}$-divisor, then by \cite[Lemma 3.5.3]{bchm},
$$\textup{B}(D+A)=\bigcap_F \textup{Supp}(F)$$
where the intersection is over all effective $\mathbb{R}$-divisor $F$ with $F\equiv D+A$.
The assertion now follows from \cite[Lemma V.1.9]{nakayama04}.
\end{proof}

The proof of Proposition \ref{proposition:nef_codim_1_2} below
makes use of the following lemma.

\begin{lemma}\label{lemma:non_nef_codimension_2}
Let $Y$ be a complex projective manifold of dimension $n \ge 2$, and let $D$ be an $\mathbb{R}$-divisor.
Suppose that $D$ is nef in codimension one, and suppose furthermore that there is 
an irreducible component $B$ of $\B_{-}(D)$ of codimension two. Let $|H|$ be a base-point-free linear system on $Y$, let $0 \le k\le n-2$ be an integer, and let 
$S$ be a complete intersection of $k$ very general elements in $|H|$.
There exists a real number $a>0$ such that
$$(D_{|S}^{2}-aB\cap S) \cdot h_{3} \cdot \cdots \cdot h_{n-k} \ge 0$$
for arbitrary codimension one nef classes $h_{3},\ldots,h_{n-k}$ on $S$.
\end{lemma}

\begin{proof}
Let $\beta_1\colon Z_1 \to Y$ be an embedded resolution of $B$, and let $\beta_2\colon Z \to Z_1$ be the 
blow-up of $Z_1$ along the strict transform $B_1$ of $B$ in $Z_1$. 
Set $\beta := \beta_1 \circ \beta_2$, and let 
$E_1,\ldots, E_r$ be the $\beta$-exceptional divisors. Suppose that $E_1=\textup{Exc}(\beta_2)$. 

Set $S_1:=\beta_1^{-1}(S)$, and $T:=\beta^{-1}(S)$. Notice that $S_1$ is an embedded resolution of $B\cap S$, and that $T \to S_1$ is the blow-up of the strict transform $B_1\cap S_1$
of $B\cap S$ in $S_1$. Set $F_i:=E_i \cap T$, and denote
by $\mu_1\colon S_1 \to S$, $\mu_2\colon T \to S_1$, and $\mu\colon T \to S$ the natural morphisms.
Set $D_S:=D_{|S}$, and let $h_{3},\ldots,h_{n-k}$ be codimension one nef classes on $S$.

Set $a_i:=\sigma_{E_i}(\beta^*D)\in \mathbb{R}_{\ge 0}$. The $\mathbb{R}$-divisor
$\displaystyle{\beta^*D-\sum_{1 \le i\le r}a_iE_i}$ is then nef in codimension one.
Notice that $a_1 >0$ by Lemma \ref{lemma:multiplicity_non_nef_locus} above.

By Lemma \ref{lemma:restriction_movable_classes} below, the restriction of 
$\displaystyle{\beta^*D-\sum_{1 \le i\le r}a_iE_i}$ to $T$ is also nef in codimension one, and therefore
$$\big(\mu^*D_S-\sum_{1 \le i\le r}a_iF_i\big)^2\cdot \mu^*h_{3} \cdot \cdots \cdot \mu^*h_{n-k} \ge 0.$$
But $F_i\cdot \mu^*D_S \cdot \mu^*h_{3} \cdot \cdots \cdot \mu^*h_{n-k} = 0$ since $F_i$ is $\mu$-exceptional,
and $F_i\cdot \mu^*h_{3} \cdot \cdots \cdot \mu^*h_{n-k} \equiv 0$ when $i \ge 2$ since
$\mu(F_i) \subsetneq B\cap S$ for each $i \ge 2$. Thus
$$\big(\mu^*D_S-\sum_{1 \le i\le r}a_iF_i\big)^2\cdot \mu^*h_{3} \cdot \cdots \cdot \mu^*h_{n-k}
=\big(\mu^*(D_S^2)+a_1^2F_1^2\big)\cdot \mu^*h_{3} \cdot \cdots \cdot \mu^*h_{n-k}.$$
The projection formula gives
$$\big(\mu^*(D_S^2)+a_1^2F_1^2\big)\cdot \mu^*h_{3} \cdot \cdots \cdot \mu^*h_{n-k}
=\big(D_S^2-a_1^2 B\cap S\big)\cdot h_{3} \cdot \cdots \cdot h_{n-k},$$
using the fact that ${\mu_2}_{*}F_1^2=-B_1\cap S_1$. 
This proves the lemma.
\end{proof}

\begin{lemma}\label{lemma:restriction_movable_classes}
Let $Y$ be a complex projective manifold, let $V\subset |H|$ be a not necessarily complete base-point-free linear system on $Y$, and let $D$ be an $\mathbb{R}$-divisor. 
If $D$ is pseudo-effective (resp. nef in codimension one), then its restriction to a very general element in
$V$ is pseudo-effective (resp. nef in codimension one) as well.
\end{lemma}

\begin{proof}
By \cite[Proposition III 1.14]{nakayama04}, an $\mathbb{R}$-divisor is nef in codimension one if and only if  it is movable. Suppose that $D$ is pseudo-effective (resp. nef in codimension one). 
There is a sequence of effective (resp. movable) integral divisors $M_i$ on $Y$ and a sequence $\lambda_i$ of non-negative real numbers such that $\lambda_i[M_i]\to [D]$ in $\Psef(Y)$ as $i\to+\infty$.
Now, observe that the restriction of an effective (resp. movable) divisor to a general element in $V$ is effective (resp. movable). So, if $H'$ is a very general element in $V$, then for each $i$, ${M_i}_{|H'}$ is effective (resp. movable), and hence so is $D_{|H'}$.
\end{proof}

The proofs of Lemma \ref{lemma:tautological_class_extremal} and Proposition \ref{proposition:negative_part} follow arguments that go back at least as far as \cite[Theorem IV 4.8]{nakayama04}.

\begin{lemma}\label{lemma:tautological_class_extremal}
Let $X$ be a complex projective manifold, and let $\sE$ be a coherent locally free sheaf on $X$. Suppose that $\sE$ is semistable with respect to any polarization and that $c_1(\sE) \equiv 0$. Set $Y:=\mathbb{P}_X(\sE)$, and denote by $\sO_Y(1)$ the tautological line bundle. If $\xi:=[\sO_Y(1)]\in \N^1(Y)$ is pseudo-effective, then it generates an extremal ray of the cone of pseudo-effective classes $\Psef(Y)$.
\end{lemma}

\begin{proof}Let $\xi_1$ and $\xi_2$ be pseudo-effective classes on $Y$ such that $\xi = \xi_1+\xi_2$.
Write $\xi_i=a_i\xi+\pi^*\gamma_i$ for some real number $a_i$ and some class $\gamma_i$ on $X$. Notice that
$a_i \ge 0$, $a_1+a_2=1$, and $\gamma_1+\gamma_2=0$.

Denote by $\pi\colon Y \to X$ the natural projection.
Let $H$ be an ample divisor, and let $C \subset X$ be a general complete intersection curve of elements in $|mH|$ where $m$ is a sufficiently large positive integer. 
By the restriction theorem of Mehta and Ramanathan,
the locally free sheaf $\sE_{|C}$ is stable with $\deg(\sE_{|C})=0$. In particular, $\sE_{|C}$ is nef. 
Set $Z:=\pi^{-1}(C)$.
By \cite[Lemma 2.2]{fulger}, 
$$\Nef(Z)=\Psef(Z)=\langle \xi_{|Z},f\rangle$$
where $f$ denotes the numerical class of a fiber of the projection morphism $Z \to C$. This implies that 
$\deg({\gamma_i}_{|C}) \ge 0$, and hence $\deg({\gamma_i}_{|C}) = 0$ since $\gamma_1+\gamma_2=0$. Since $H$ is arbitrary, we conclude that $\gamma_1=\gamma_2=0$. This completes the proof of the lemma.
\end{proof}

\begin{prop}\label{proposition:negative_part}
Let $X$ be a complex projective manifold, and let $\sE$ be a locally free sheaf on $X$. Suppose that $\sE$ is semistable with respect to an ample divisor $H$, and that $\mu_H(\sE)=0$. Set $Y:=\mathbb{P}_X(\sE)$.
Suppose that the tautological line bundle $\sO_Y(1)$
is pseudo-effective. If $\sO_Y(1)$ is not nef in codimension one, then there exists a line bundle $\sL$ with 
$\mu_H(\sL)=0$ and a positive integer $m$ such that $h^0\big(X,(S^m\sE)\otimes\sL\big)\neq 0$.
\end{prop}

\begin{proof}
Denote by $\Xi$ a tautological divisor on $Y$, and by $\pi\colon Y \to X$ the natural morphism.
Let $\Xi=P+N$ be the divisorial Zariski decomposition of $\Xi$. 
Write $N=\sum_{i\in I}\sigma_i N_i$, where $N_i$ is a prime exceptional divisor, and $\sigma_i \in \mathbb{R}_{> 0}$. We have $N_i \sim_\mathbb{Z} m_i\Xi+\pi^*\Gamma_i$
for some divisor $\Gamma_i$ on $X$ and some non-negative integer $m_i$.

Let $C \subset X$ be a complete intersection curve of very general elements in $|mH|$ where $m$ is a sufficiently large positive integer. 
By the restriction theorem of Mehta and Ramanathan,
the locally free sheaf $\sE_{|C}$ is stable with $\deg(\sE_{|C})=0$. 
Set $Z:=\pi^{-1}(C)$. Notice that $\Xi_{|Z}$ is pseudo-effective by Lemma \ref{lemma:restriction_movable_classes}. 
By Lemma \ref{lemma:tautological_class_extremal}, 
$[{\pi^*\Gamma_i}_{|Z}]\in \mathbb{R}_{\ge 0}[\Xi_{|Z}]$ for each $i\in I$, and thus
$\Gamma_i\cdot H^{n-1}=0$, where $n:=\dim X$. 
Pick $i\in I$. If $m_i=0$, then $h^0\big(X,\sO_X(\Gamma_i)\big)\neq 0$, and hence $\Gamma_i\sim_\mathbb{Z}0$ since 
$\Gamma_i\cdot H^{n-1}=0$. This implies that $N_i=0$, yielding a contradiction. Therefore, $m_i>0$ and 
$h^0\big(X,(S^{m_i}\sE)\otimes \sO_X(\Gamma_i)\big)\neq 0$. This proves the proposition.
\end{proof}

The proof of the next result follows the line of argument given in \cite[Theorem 7.6]{bdpp}.

\begin{lemma}\label{lemma:restricted_base_locus_dimension_one}
Let $S$ be a smooth complex projective surface, and let $\sE$ be a locally free sheaf of rank $r\ge 2$ on $S$. Suppose that 
$\sE$ is semistable with respect to an ample divisor $H$, and that $c_1(\sE)\cdot H = 0$. 
Let $\xi \in \N^1(Y)$ be the class of the tautological line bundle $\sO_Y(1)$.
If any irreducible component of $\textup{B}_{-}(\xi)$ has dimension at most $1$,
then $c_1(\sE)^2=c_2(\sE)=0$.
\end{lemma}

\begin{proof}
Denote by $\pi \colon Y \to S$ the natural morphism, and denote by $h \in \N^1(S)$ the class of $H$.
Let $G\subset Y$ be a very general hyperplane section. 
Suppose that $[G] \equiv m(\xi +t\pi^*h)$ for some positive integers $m$ and $t$.
Notice that $G$ does not contain any irreducible
component of $\textup{B}_{-}(\xi)$. It follows that $\xi_{|G}$ is nef, and hence 
$$\xi^{r}\cdot G \ge 0.$$ 
The equation
$$\xi^r\equiv \pi^*c_1(\sE)\cdot \xi^{r-1} -\pi^*c_2(\sE)\cdot \xi^{r-2},$$ yields 
$$\xi^{r}\cdot G=m\big(\pi^*c_1(\sE)\cdot \xi^{r-1}-\pi^*c_2(\sE)\cdot \xi^{r-2}\big)\cdot (\xi +t\pi^*h)=
m \big(c_1(\sE)^2-c_2(\sE)\big),$$
and hence $$c_1(\sE)^2 - c_2(\sE) \ge 0.$$
On the other hand, we have $$2rc_2(\sE)-(r-1)c_1(\sE)^2\ge 0$$ by Bogomolov's inequality (see \cite[Theorem 3.4.1]{HuyLehn}), and thus $c_1(\sE)^2 \ge 0$. Finally, the Hodge index theorem implies that  $c_1(\sE)^2 \le 0$, and hence
we must have $c_1(\sE)^2=c_2(\sE)=0$. This proves the lemma.
\end{proof}

\begin{say}[The holonomy group of a stable reflexive free sheaf]
Let $X$ be a normal complex projective variety, and let $\sE$ be a reflexive sheaf on $X$. 
Suppose that $\sE$ is stable with respect to an ample Cartier divisor $H$.
For a sufficiently large positive integer $m$,
let $C \subset X$ be a general complete intersection curve of elements in $|mH|$. Let $x \in C$.
By the restriction theorem of Mehta and Ramanathan,
the locally free sheaf $\sE_{|C}$ is stable with $\deg(\sE_{|C})=0$, and hence
it corresponds to a unique unitary representation $\rho\colon\pi_1(C,x)\to\mathbb{U}(\sE_x)$ 
by a result of Narasimhan and Seshadri (\cite{narasimhan_seshadri65}). The \emph{holonomy group} $\Hol_x(\sE)$
of $\sE$ is the Zariski closure of $\rho\big(\pi_1(C,x)\big)$ in $\textup{GL}(\sE_x)$. It does not depend on $C \ni x$ provided that $m$ is large enough. Moreover, the fiber map $\sE \to \sE_x$ induces a one-to-one correspondance between direct summands of $\sE^{\otimes r}\boxtimes (\sE^*)^{\otimes s}$ and 
$\Hol_x(\sE)$-invariant subspaces of $\sE_x^{\otimes r}\otimes (\sE_x^*)^{\otimes s}$, where $r$ and $s$ are non-negative integer (see \cite[Theorem 1]{balaji_kollar}).
\end{say}

The proof of Theorem \ref{theorem:stable_sheaf_psef} makes use of the following lemma. Example \ref{example:hol_connected} below shows that the statement of \cite[Lemma 40]{balaji_kollar} is slightly incorrect. An extra assumption is needed to guarantee that the holonomy groups are well-defined.

\begin{lemma}[{\cite[Lemma 40]{balaji_kollar}}]\label{lemma:holonomy_group_connected}
Let $X$ be a normal complex projective variety, let $x \in X$ be a general point, and let $\sE$ be a reflexive sheaf on $X$.
Suppose that $\sE$ is stable with respect to an ample divisor $H$, and that $\mu_H(\sE)=0$. Suppose furthermore that, for any finite morphism $f\colon \wt X \to X$ that is \'etale in codimension one, the 
reflexive pull-back $f^{[*]}\sE$ is stable with respect to $f^*H$. Then there exists a finite morphism $f\colon \wh X \to X$, \'etale in codimension one, such that $\Hol_{\wh x}(f^{[*]}\sE)$ is connected, where 
$\wh x$ is a point on $\wh X$ such that $f(\wh x)=x$.
\end{lemma}

\begin{exmp}[{see \cite[Example 8.6]{gkp_bo_bo}}]\label{example:hol_connected}
Let $Z$ be a projective $\textup{K}3$-surface, let $\wt X:=Z \times Z$, and let $\iota\in\textup{Aut}(\wt X)$
be the automorphism which interchanges the two factors. The
quotient $X := \wt X / \iota$ is then a normal projective 
variety. 
The quotient map $\pi \colon \wt X \to X$ is finite and \'etale in codimension one. 
The tangent sheaf $T_X$ of $X$ is stable with respect to any ample polarization on $X$ (see \cite[Example 8.6]{gkp_bo_bo}). Let $x$ is a general point on $X$. Then $\Hol_x(T_X)^\circ =\textup{SL}_2(\mathbb{C})$, and 
$\Hol_x(T_X)/\Hol_x(T_X)^\circ \cong\mathbb{Z}/2\mathbb{Z}$. Moreover, the morphism
$\pi$ is the map given by \cite[Lemma 40]{balaji_kollar}. But, the reflexive pull-back $\pi^{[*]}T_X = T_{\wt X}=T_Z\boxplus T_Z$ is obviously not stable.
\end{exmp}

\begin{lemma}\label{lemma:holonomy_group_versus_strong_stability}
Let $X$ be a complex projective manifold, let $x \in X$, and let $\sE$ be a coherent locally free sheaf on $X$.
Suppose that $\sE$ is stable with respect to an ample divisor $H$, and that $\mu_H(\sE)=0$.
Suppose furthermore that its holonomy group $\Hol_x(\sE)$ is connected. 
Then, for any finite cover $f \colon \wt X \to X$ with $\wt X$ smooth and projective,
the pull-back $f^*\sE$ is stable with respect to $f^*H$. 
\end{lemma}

\begin{proof}
Let $\wt X$ be a complex projective manifold, and let $f\colon \wt X \to X$ be a finite cover.
By \cite[Theorem 1]{kempf}, the locally free sheaf $f^*\sE$ is polystable with respect to $f^*H$.
For a sufficiently large positive integer $m$,
let $C \subset X$ (resp. $\wt C$) be a general complete intersection curve of elements in $|mH|$
(resp. $|mf^*H|$), and 
let $x \in C$ (resp. $\wt x \in f^{-1}(x))$.
By the restriction theorem of Mehta and Ramanathan,
the locally free sheaf $\sE_{|C}$ is stable with $\deg(\sE_{|C})=0$, 
and hence it corresponds to a unique unitary irreducible representation $\rho\colon\pi_1(C,x)\to\mathbb{U}(\sE_x)$
by \cite{narasimhan_seshadri65}. Notice that $f^*\sE_{|\wt C}$
is then induced by the representation
$\wt\rho:=\rho\circ\pi_1(f_{|\wt C})\colon\pi_1(\wt C,\wt x)\to \pi_1(C,x)\to\mathbb{U}(\sE_x)\cong \mathbb{U}\big((f^*\sE)_{\wt x}\big)$.

We argue by contradiction and assume that $f^*\sE$ is not stable with respect to $f^*H$. So let 
$\sG$ be a $f^*H$-stable direct summand of $f^*\sE$. Then $\sG_{\wt x}$ is $\wt\rho$-invariant, and since 
the image of 
$\pi_1(f_{|\wt C})\colon\pi_1(\wt C,\wt x)\to \pi_1(C,x)$ has finite index,
the orbit $\pi_1(C,x)\cdot\sG_{\wt x}$ of $\sG_{\wt x}$
is a finite union of proper linear subspaces, where we view $\sG_{\wt x}$ as a linear subspace of $\sE_x\cong (f^*\sE)_{\wt x}$. 
It follows that $\pi_1(C,x)\cdot\sG_{\wt x}$ is also $\Hol_x(\sE)$-invariant.
Now, since $\Hol_x(\sE)$ is a connected algebraic group, we conclude that 
$\sG_{\wt x}$ is $\pi_1(C,x)$-invariant. Therefore $\sE$ is not $H$-stable, yielding a contradiction. This completes the proof of the lemma.
\end{proof}

We now provide another technical tool for the proof of Theorem \ref{theorem:stable_sheaf_psef}.

\begin{prop}\label{proposition:nef_codim_1_2}
Let $X$ be a complex projective manifold, and let $\sE$ be a rank $3$ locally free sheaf on $X$. Suppose that $\sE$ is stable with respect to an ample divisor $H$, and that $\mu_H(\sE)=0$. 
Suppose furthermore that $\Hol_x(E)$ is connected for $x\in X$.
Set $Y:=\mathbb{P}_X(\sE)$, and denote by $\sO_Y(1)$ the tautological line bundle. 
Suppose that $\xi:=[\sO_Y(1)]\in \N^1(Y)$ is pseudo-effective. If $\xi$ is nef in codimension one, then it is nef in codimension two.
\end{prop}

\begin{proof}
Denote by $\pi\colon Y \to X$ the natural projection.
We argue by contradiction and assume that 
$\xi$ is nef in codimension one, and that
there exists an irreducible component $B$ of $\B_{-}(\xi)$ of codimension two.

For a sufficiently large positive integer $m$,
let $C \subset X$ be a complete intersection curve of very general elements in $|mH|$.
By the restriction theorem of Mehta and Ramanathan,
the locally free sheaf $\sE_{|C}$ is stable with $\deg(\sE_{|C})=0$. 
Moreover, if $x\in C$, then $\Hol_x(\sE_{|C})=\Hol_x(\sE)$ is connected.
Set $Z:=\pi^{-1}(C)$, and $\xi_Z:=\xi_{|Z}$. 
Notice that $\dim Z=3$.
Then $\xi_Z$ is nef,
and $\xi_{Z}^{3}=0$.
By Lemma \ref{lemma:non_nef_codimension_2}, there exists a 
real number $a>0$ such that
$0 \le a(B\cap Z) \cdot \xi_Z \le \xi_Z^3=0$, and hence
$(B\cap Z) \cdot \xi_Z=0$. In particular, since $\xi_Z$ is relatively ample over $C$,  
any irreducible component of $B\cap Z$ maps onto $C$. Let $\wt C$ be the normalization of $B\cap Z$, and denote by $f\colon \wt C \to C$ the induced morphism. The natural morphism $g\colon \wt C \to Z$ induces a surjective morphism $f^*(\sE_{|C}) \onto g^*(\sO_Y(1)_{|Z})$. Since $\deg g^*(\sO_Y(1)_{|Z})=(B\cap Z)\cdot \xi_Z=0$, the locally free sheaf $f^*(\sE_{|C})$ is not stable. But this contradicts Lemma \ref{lemma:holonomy_group_versus_strong_stability}, completing the proof of the proposition.
\end{proof}

\begin{proof}[Proof of Theorem \ref{theorem:stable_sheaf_psef}]
We maintain notation and assumptions of Theorem \ref{theorem:stable_sheaf_psef}.
We claim that we may assume without loss of generality that the algebraic group $\Hol_x(\sE)$ is connected.
Indeed, by Lemma \ref{lemma:holonomy_group_connected}, 
there exists a finite cover $f\colon \wt X \to X$, \'etale in codimension one, such that $\Hol_{\wt x}(f^{[*]}\sE)$ is connected, where 
$\wt x$ is a point on $\wt X$ such that $f(\wt x)=x$.
By the Nagata-Zariski purity theorem, $f$ is \'etale in codimension two, and hence $\wt X$ is 
smooth in codimension two as well. 
Suppose now that the conclusion of Theorem \ref{theorem:stable_sheaf_psef} holds for 
$f^{[*]}\sE$. The sheaf 
$S^{[i]}\sE\otimes\sO_X(jH)$ is a direct summand of 
$f_{[*]} \big (S^{i} (f^{*} \sE)\otimes f^*\sO_X(jH)\big)$ for any non-negative integers $i$ and $j$. Thus, if 
$f^{[*]}\sE$ satisfies condition (1) in the statement of Theorem \ref{theorem:stable_sheaf_psef}, then the same holds for $\sE$.
Let $X^\circ \subset X_{\textup{reg}}$ be the maximal open set where $\sE$ is locally free. 
Then $X^\circ$ has codimension at least 3 in $X$ by \cite[Corollary 1.4]{hartshorne80} using the fact that $X$ is smooth in codimension two. Since
$c_1\big(f_{|f^{-1}(X^\circ)}^*\sE_{|X^\circ}\big)=f_{|f^{-1}(X^\circ)}^*c_1(\sE_{|X^\circ})$ and
$c_2\big(f_{|f^{-1}(X^\circ)}^*\sE_{|X^\circ}\big)=f_{|f^{-1}(X^\circ)}^*c_2(\sE_{|X^\circ})$, we conclude that
$c_1(f^{[*]}\sE)^2 \cdot (f^*H)^{n-2}= \deg (f) c_1(\sE)^2 \cdot H^{n-2}$ and that
$c_2(f^{[*]}\sE) \cdot (f^*H)^{n-2}= \deg (f) c_2(\sE) \cdot H^{n-2}$. This implies that if 
$f^{[*]}\sE$ satisfies condition (2) in the statement of Theorem \ref{theorem:stable_sheaf_psef}, then the same holds for $\sE$. Finally, if $f^{[*]}\sE$ satisfies condition (3) in the statement of Theorem \ref{theorem:stable_sheaf_psef}, then the same obviously holds for $\sE$.
Thus, by replacing $X$ with $\wt X$, we may assume that $\Hol_x(\sE)$ is connected. This proves our claim.

Suppose from now on that $h^0\big(X,S^{[i]}\sE\otimes\sO_X(jH)\big) \neq 0$ for infinitely many 
$(i,j) \in \mathbb{N} \times \mathbb{N}_{\ge 1}$ with $i/j \to +\infty$.
Let $S$ be a smooth two dimensional complete intersection of very general elements in $|mH|$ for a sufficiently large positive integer $m$. Observe that $S$ is contained in the smooth locus $X_\textup{reg}$ of $X$, and that 
$\sE$ is locally free along $S$ (see \cite[Corollary 1.4]{hartshorne80}).
We may also obviously assume that $x \in S$.
Set $\sE_S:=\sE_{|S}$, and $H_S:=H_{|S}$. 
By the restriction theorem of Mehta and Ramanathan,
the locally free sheaf $\sE_S$ is stable with respect to $H_S$, and $\mu_{H_S}(\sE_S)=0$. Moreover, the algebraic group
$\Hol_x(\sE_S)=\Hol_x(\sE)$ is connected.
Set $Y:=\mathbb{P}_S(\sE_S)$ with 
natural morphism $\pi\colon Y \to S$.
Denote by $\xi\in\N^1(Y)$ the numerical class of the tautological line bundle $\sO_Y(1)$.
Notice that $h^0\big(Y,S^{i}\sE_S\otimes\sO_S(jH_S)\big) \neq 0$ for infinitely many 
$(i,j) \in \mathbb{N} \times \mathbb{N}_{\ge 1}$ with $i/j \to +\infty$. This implies that $\xi \in \Psef(Y)$.
If $r=1$, then obviously we must have $\xi=0$. Suppose from now on that $r\in\{2,3\}$.

\medskip

\noindent\textit{Case 1: $\xi$ is nef in codimension one.} If $r=3$, then $\xi$ is automatically nef in codimension two by Proposition \ref{proposition:nef_codim_1_2}. In either case, any irreducible component of 
$\textup{B}_{-}(\xi)$ has dimension at most $1$.
Thus, by Lemma \ref{lemma:restricted_base_locus_dimension_one}, we must have 
$c_1(\sE)^2\cdot H^{n-2}=c_1(\sE_S)^2=0$
and $c_2(\sE)\cdot H^{n-2}=c_2(\sE_S)=0$.

\medskip

\noindent\textit{Case 2: $\xi$ is not nef in codimension one.} By Proposition \ref{proposition:negative_part},
there exists a line bundle $\sL_S$ on $S$ with 
$\mu_{H_S}(\sL_S)=0$ and a positive integer $k$ such that $h^0(S,S^k\sE_S\otimes\sL_S)\neq 0$.
Since $S^k\sE_S$ is polystable, $\sL_S^{\otimes -1}$ is a direct summand  of $S^k\sE_S$, and hence
$(\sL_S^{\otimes -1})_x \subset (S^k\sE)_x$ is $\Hol_x(\sE)$-invariant. 

Suppose first that $r=2$. By \cite[45]{balaji_kollar}, $\Hol_x(\sE)=\textup{SL}(\sE_x)$ or
$\textup{GL}(\sE_x)$. In either case, $(S^k\sE)_x$ is an irreducible 
$\Hol_x(\sE)$-module, yielding a contradiction.

Suppose now that $r=3$. Then $\Hol_x(\sE)=\textup{SL}(\sE_x)$, 
$\textup{GL}(\sE_x)$, $\textup{SO}(\sE_x)$, or $\textup{GSO}(\sE_x)$ by \cite[45]{balaji_kollar} again, where
$\textup{GSO}(\sE_x)$ denotes the group of proper similarity transformations.
Arguing 
as above, we conclude that $\Hol_x(\sE)=\textup{SO}(\sE_x)$ or $\textup{GSO}(\sE_x)$. 
In either case, 
there exists a rank one reflexive sheaf $\sL$ on $X$ with $\mu_H(\sL)=0$
such that $h^0(X,S^{2}\sE\boxtimes\sL)\neq 0$. 
This completes the proof of the theorem. 
\end{proof}

\section{Holomorphic Riemannian metric and holomorphic connection}

\begin{say}[Bott connection]\label{partial_connection}
Let $X$ be a complex manifold, let $\sG\subset T_X$ be a regular foliation, and set $\sN=T_X/\sG$. Let $p\colon T_X\to \sN$ denotes the natural projection. For sections $U$ of $\sN$, $T$ of $T_X$, and $V$ of $\sG$ over some open subset of $X$ with 
$U=p(T)$, set $\nabla^{\textup{B}}_V U=p([V,U])$. This expression is well-defined,
$\sO_X$-linear in $V$ and satisfies the Leibnitz rule 
$\nabla^{\textup{B}}_V(fU)=f\nabla^{\textup{B}}_V U+(V\cdot f)U$ so that $\nabla^{\textup{B}}$ is an $\sG$-connection on $\sN$ (see \cite{baum_bott70}). We refer to it as the \textit{Bott} (partial) \emph{connection} on $\sN$.
\end{say}

\begin{say}[Holomorphic Riemannian metric]
Given a complex manifold $X$ and a vector bundle $\sE$ on $X$,
recall that a holomorphic metric $\mathfrak{g}$ on $\sE$ is a global section of $S^2(\sE^*)$
such that $\mathfrak{g}(x)$ is non-degenerate for all $x \in X$.
\end{say}

\begin{lemma}\label{lemma:levi_civita}
Let $X$ be a complex manifold, and let $T_X = \sE \oplus \sG$ be a decomposition of $T_X$ into locally free sheaves. Suppose that $\sE$ is involutive, and suppose furthermore that $\sE$ admits a holomorphic metric 
$\mathfrak{g}$. Then there exists an $\sE$-connection $\nabla^{\textup{LC}}$ on $\sE$ such that, for sections 
$U$, $V$, and $W$ of $\sE$ over some open subset of $X$, the following holds: 
\begin{enumerate}
\item $\nabla^{\textup{LC}}_U V-\nabla^{\textup{LC}}_V U = [U,V]$ $(\nabla^{\textup{LC}}$ is torsion-free$)$, and
\item $W \cdot  \mathfrak{g}(U,V) = \mathfrak{g}\big(\nabla^{\textup{LC}}_W U,V\big) + \mathfrak{g}\big(U,\nabla^{\textup{LC}}_W V\big)$ $(\nabla^{\textup{LC}}$ preserves $\mathfrak{g})$.
\end{enumerate}
\end{lemma}

\begin{defn}
We will refer to $\nabla^{\textup{LC}}$ as the \emph{Levi-Civita} (partial) \emph{connection} on $\sE$.
\end{defn}

\begin{proof}[Proof of Lemma \ref{lemma:levi_civita}]
Let $U$ and $V$ be local sections of $\sE$ over some open subset of $X$. Then 
$\nabla^{\textup{LC}}_U V$ is defined by 
$$
2\mathfrak{g}\big(\nabla^{\textup{LC}}_U V,\bullet\big) =
U\cdot \mathfrak{g}(V,\bullet)+V\cdot \mathfrak{g}(\bullet,U)-\bullet\cdot \mathfrak{g}(U,V)+\mathfrak{g}\big([U,V],\bullet\big)
-\mathfrak{g}\big([V,\bullet],U\big)-\mathfrak{g}\big([U,\bullet],V\big).
$$
A straightforward local computation shows that $\nabla^{\textup{LC}}$ is a torsion-free $\sE$-connection on $\sE$
that preserves the metric.
\end{proof}

\begin{prop}\label{proposition:holomorphic_connection}
Let $X$ be a complex manifold, and let $T_X = \sE \oplus \sG$ be a decomposition of $T_X$ into locally free involutive subsheaves. Suppose that $\sE$ admits a holomorphic metric. Then $\sE$ has a holomorphic connection.
\end{prop}

\begin{proof}
Let $X=U+V$ and $W$ be local sections of $T_X = \sE \oplus \sG$ and $\sE$ respectively.
Set $$\nabla_X W : = \nabla^{\textup{LC}}_U W + \nabla^{\textup{B}}_V W$$
where $\nabla^{\textup{LC}}$ denotes the Levi-Civita connection on $\sE$ and where
$\nabla^{\textup{B}}$ denotes the Bott connection on $\sE$ induced by the foliation $\sG \subset T_X$.
This expression is $\sO_X$-linear in $X$ and satisfies the Leibnitz rule 
$\nabla_X(fW)=f\nabla_X W+(X\cdot f)W$ so that $\nabla$ is a holomorphic connection on $\sE$.
\end{proof}

\begin{rem}\label{remark:atiyah}
In the setup of Proposition \ref{proposition:holomorphic_connection}, let $S \subset X$ be a projective subvariety. Then characteristic classes of $\sE_{|S}$ vanish. This follows from \cite{atiyah57}.
\end{rem}

\section{Bost-Campana-P\u{a}un algebraicity criterion $-$ Algebraicity of leaves, II}

We prove Theorem \ref{theorem:pereira_touzet_conjecture} in this section.   
We first provide a technical tool for the proof of our main result (see \cite[Proposition 6.1]{druel15} for a somewhat related result).

\begin{prop}\label{proposition:alg_int_reduction_strongly_stable_case}
Let $X$ be a normal complex projective variety, and let $H$ be an ample Cartier divisor. Let 
$\sE\subset T_X$ be a foliation on $X$. Suppose that $\sE$ is $H$-stable and that $\mu_H(\sE)=0$.
Suppose furthermore that through a general point of $X$, there is a
positive-dimensional algebraic subvariety that is tangent to $\sE$. Then $\sE$ has algebraic leaves.
\end{prop}

\begin{proof}
There exist  a normal projective variety $Y$, unique up to birational equivalence,  a dominant rational map with connected fibers $\varphi:X\map Y$,
and a holomorphic foliation $\sH$ on $Y$ such that the following holds (see \cite[Section 2.4]{loray_pereira_touzet}):
\begin{enumerate}
\item $\sH$ is purely transcendental, i.e., there is no positive-dimensional algebraic subvariety through a general point of $Y$ that is tangent to $\sG$; and
\item $\sE$ is the pullback of $\sH$ via $\varphi$.
\end{enumerate}
Let $\sG \subseteq \sE$ be the foliation on $X$ induced by $\phi$.
Let $\psi\colon Z\to Y$ be the family of leaves of $\sG$, and let $\beta\colon Z \to X$ be the natural morphism, so that 
$\phi:=\psi \circ \beta^{-1}\colon X \dashrightarrow Y$.
By \ref{family_leaves},
there is an effective divisor $R$ on $X$ such that
$$K_\sG-K_\sE=-(\phi^*K_\sH+R).$$
Notice that the pull-back $\phi^*K_\sH$ is well-defined (see Definition \ref{definition:pull-back}).
Applying \cite[Corollary 4.8]{campana_paun15} to the foliation induced by $\sH$ on a desingularization of $Y$, we see that $K_{\sH}$ is pseudo-effective: given an ample divisor $A$ on $Y$ and a positive number $\varepsilon \in\mathbb{Q}$, there exists an effective
$\mathbb{Q}$-divisor $D_\varepsilon$ such that $K_\sH+\varepsilon A\sim_\mathbb{Q} D_\varepsilon$.
This implies that $\mu_H(\phi^*K_\sH) \ge 0$, and hence $\mu_H(K_\sG) \le 0$. Since $\sE$ is $H$-stable, we must have $\sG=\sE$. This proves the lemma.
\end{proof}

The proof of Theorem \ref{theorem:pereira_touzet_conjecture}
relies on a criterion (see  Proposition \ref{proposition:alg_integrability}) that 
guarantees that a given foliation has algebraic leaves, which we establish now.

The following observation, due to Bost, will prove to be crucial.

\begin{prop}[{\cite[Proposition 2.2]{bost_germs}}]\label{proposition:multiplicity_versus_algebraic_integrability}
Let $Z$ be a projective variety over a field $k$, let $x$ be a $k$-point, and let $H$ be an ample divisor.
Let $\wh V \subset \wh Z$ be a smooth formal subscheme of the formal completion $\wh Z$ of $Z$ at $x$. Then 
$\wh V$ is algebraic if and only if there exists $c>0$ such that,
for any positive integer $j$ and any section $s \in H^0\big(Z,\sO_Z(jH)\big)$ such that $s_{|\wh V}$ is non-zero, the multiplicity $\mult_{x} (s_{|\wh V})$ of $s_{|\wh V}$ at $x$ is $\le c j$.
\end{prop}

\begin{cor}\label{corollary:multiplicity_versus_algebraic_integrability}
Let $Y^\circ$ be a smooth complex quasi-projective variety, let $Z$ be a complex projective variety
with $Y^\circ\subseteq Z$, and let $H$ be an ample Cartier divisor on $Z$. 
Let $V \subset Z$  be a germ of smooth locally closed
analytic submanifold along $Y^\circ$ in $Z$. Then $V$ is algebraic if and only if there exists $c>0$ such that,
for any positive integer $j$ and any section $s \in H^0\big(Z,\sO_Z(jH)\big)$ such that $s_{|V}$ is non-zero, 
the multiplicity $\mult_{Y^\circ} (s_{|V})$ of $s_{|V}$ along $Y^\circ$ is 
$\le c j$.
\end{cor}

\begin{proof}
Let $\eta \in Z$ be the generic point of $Y^\circ$. Denote by $k(\eta)$ its residue field.
Set $Z_\eta : = Z \otimes k(\eta)$, and $H_\eta:=H \otimes k(\eta)$ . 
Notice that $H^0\big(Z_\eta,\sO_{Z_\eta}(jH_\eta)\big)\cong H^0\big(Z,\sO_Z(jH)\big)\otimes k(\eta)$ for any number $j \in \mathbb{Z}$.
The point $\eta$ corresponds to a $k(\eta)$-point on $Z_\eta$ still denoted by $\eta$. Let 
$\wh V$ be the formal completion of $V$ along $Y^\circ$. Then $\wh V$ induces
a smooth formal subscheme $\wh V_\eta$ of the formal completion  $\wh Z_\eta$ of $Z_\eta$ at $\eta$. Observe that $\wh V$ is algebraic if and only if $\wh V_\eta$ is.
The lemma now follows from 
Proposition \ref{proposition:multiplicity_versus_algebraic_integrability}
applied to $(Z_\eta,\eta,H_\eta)$ and $\wh V_\eta$
since $\mult_{Y^\circ} (s_{|V}) = \mult_{\eta} \big({s\otimes k(\eta)}_{|\wh V_\eta}\big)$ for any number $j\in \mathbb{Z}$ and any section $s \in H^0\big(Z,\sO_Z(jH)\big)$.
\end{proof}

The proof of Proposition \ref{proposition:alg_integrability} below follows the line of argument given in \cite[4.1]{campana_paun15} (see also \cite[Corollary 3.8]{bost}).

\begin{prop}\label{proposition:alg_integrability}
Let $X$ be a normal complex projective variety, 
let $H$ be an ample Cartier divisor, and let $\sG\subseteq T_X$ be a foliation.
Suppose that there exists $c>0$ such that
$h^0\big(X,S^{[i]}\sG^*\otimes\sO_X(jH)\big)=0$
for any positive integer $j$ and any natural number $i$ satisfying 
$i>cj$. Then $\sG$ has algebraic leaves.
\end{prop}

\begin{proof}
Let $X^\circ \subset X_{\textup{reg}}$ be the maximal open set where $\sG_{|X_{\textup{reg}}}$ is a subbundle of 
$T_{|X_{\textup{reg}}}$. Set $Z^\circ := X^\circ\times X^\circ$,
$Z:=X \times X$, and $A:=p_1^*H+p_2^*H$ where $p_1,p_2\colon Z=X\times X \to X$ denote the projections on $X$. Let $Y^\circ \subset Z^\circ$ be the open subset of the diagonal corresponding to $X^\circ$.
Applying Frobenius' Theorem to the regular foliation
$$({p_1}_{|Z^\circ})^*\sG_{|X^\circ}\subset ({p_1}_{|Z^\circ})^*T_{X^\circ} \subset 
({p_1}_{|Z^\circ})^*T_{X^\circ}\oplus ({p_2}_{|Z^\circ})^*T_{X^\circ} = T_{Z^\circ}$$ 
on $Z^\circ$, we see that there exists a smooth locally closed
analytic submanifold $V \subset Z^\circ$
containing $Y^\circ$ such that ${p_2}_{|V}$ is smooth, and such that
its fibers are analytic open subsets of the leaves of the foliation 
$({p_1}_{|Z^\circ})^*\sG_{|X^\circ} \subset T_{Z^\circ}$
passing through points of $Y^\circ$. 
Moreover, we have $\sN_{Y^\circ/V}\cong \sG_{|X^\circ}$. The closed subset 
$X \setminus X_\circ$ has codimension $\ge 2$, and hence
$h^0\big(X^\circ,S^{i}{\sN_{Y^\circ/V}^*}\otimes\sO_{X^\circ}(jH)\big)=0$
for any positive integer $j$ and any natural number $i$ satisfying 
$i>cj$ by assumption. This implies that $\mult_{Y^\circ} (s_{|V}) \le cj/2$ 
for any positive integer $j$ and any section $s \in H^0\big(Z,\sO_Z(jA)\big)$ such that $s_{|V}$ is non-zero.
The proposition now follows from Corollary \ref{corollary:multiplicity_versus_algebraic_integrability} applied to $(Z,Y^\circ,A)$ and $V$.
\end{proof}

\begin{proof}[Proof of Theorem \ref{theorem:pereira_touzet_conjecture}]
Maintaining notation and assumptions of Theorem \ref{theorem:pereira_touzet_conjecture}, set
$r:=\rank \,\sE$. 

Suppose that there exists a finite cover $f\colon \wt X \to X$ that is \'etale in codimension one such that 
the reflexive pull-back $f^{[*]}\sE$ is not stable with respect to $f^*H$.
Applying \cite[Lemma 3.2.3]{HuyLehn}, we see that the 
$f^{[*]}\sE$ is polystable, and hence, there exist non-zero reflexive sheaves $(\sE_i)_{i \in I}$, $f^*H$-stable with slopes 
$\mu_{f^*H}(\sE_i)=\mu_{f^*H}(f^{[*]}\sE)=0$ such that 
$f^{[*]}\sE\cong\oplus_{i \in I}\sE_i$.
Suppose that the number of direct summands is maximal.
Then, for any finite cover $g\colon \wh X \to \wt X$ that is \'etale in codimension one, 
the reflexive pull-back $g^{[*]}\sE_i$ is obviously stable with respect to $(f\circ g)^*H$.
Notice that $\wt X$ is still smooth in codimension two, since $f$ branches only over the singular set of $X$.
It follows from Proposition \ref{proposition:alg_int_reduction_strongly_stable_case} that if $\sE_i$ has algebraic leaves for some $i \in I$, then so does $\sE$.

\medskip

Suppose first that there exists $i_0 \in I$ such that 
$c_1(\sE_{i_0})^2\cdot (f^*H)^{n-2}\neq 0$ or
$c_2(\sE_{i_0})\cdot (f^*H)^{n-2}\neq 0$.
Applying Theorem \ref{theorem:stable_sheaf_psef} to $\sE_{i_0}^{\,*}$, we see that one of the following holds.
\begin{enumerate}
\item Either there exists $c>0$ such that
$h^0\big(\wt X,S^{[i]}\sE_{i_0}^{\,*}\otimes\sO_{\wt X}(jf^*H)\big)=0$
for any positive integer $j$ and any natural number $i$ satisfying 
$i>cj$,
\item or $r=3$, and there exists a finite morphism $g\colon \wh X \to \wt X$ that is \'etale in codimension one,  
and a rank $1$ reflexive sheaf $\sL$ on $\wh X$ with $\mu_{(f\circ g)^*H}(\sL)=0$
such that 
$h^0\big(\wh X,(S^{2}(f\circ g)^{*}\sE_{i_0}^{\,*})\boxtimes\sL\big)\neq 0$.
\end{enumerate}

If we are in case (1), apply Proposition \ref{proposition:alg_integrability} to conclude that $\sE_{i_0}$ has algebraic leaves.

Suppose that we are in case (2). Then we may assume that $\wt X = X$, and that $\sE_{i_0}=\sE$.
Taking pull-backs and double duals, we obtain a decomposition
$T_{\wh X}\cong g^{[*]}\sE \oplus g^{[*]}\sG$ into involutive subsheaves, where $g^{[*]}\sE$ is $g^*H$-stable
and $\det(g^{[*]}\sE)\cong\sO_{\wh X}$.
Applying Proposition \ref{proposition:criterion_flatness_holonomy} to $g^{[*]}\sE$, we see that
$c_1(\sE)^2\cdot H^{n-2}=\frac{1}{\deg(g)}c_1(g^{[*]}\sE)^2\cdot (g^*H)^{n-2}=0$
and $c_2(\sE)\cdot H^{n-2}=\frac{1}{\deg(g)}c_2(g^{[*]}\sE)\cdot (g^*H)^{n-2}=0$,
yielding a contradiction.

\medskip

Finally, suppose that $c_1(\sE_{i_0})^2\cdot (f^*H)^{n-2} = c_2(\sE_{i})\cdot (f^*H)^{n-2} = 0$ for each $i \in I$.
Let $S$ be a smooth two dimensional complete intersection of general elements in $|mH|$ for a sufficiently large positive integer $m$. Observe that $S$ is contained in the smooth locus $X_\textup{reg}$ of $X$, and that 
$\sE$ is locally free along $S$ (see \cite[Corollary 1.4]{hartshorne80}).
By the restriction theorem of Mehta and Ramanathan,
the locally free sheaf ${\sE_i}_{|S}$ is stable with respect to $H_{|S}$ with $\mu_{H_{|S}}({\sE_i}_{|S})=0$, and
$c_1({\sE_i}_{|S})^2=c_2({\sE_i}_{|S})=0.$ This implies that ${\sE_i}_{|S}$ is flat (see Remark \ref{remark:flatness}). A straightforward computation then shows that 
$c_2(\sE)\cdot H^{n-2}=c_2(\sE_{|S})=0$, yielding a contradiction.
This completes the proof of the theorem.
\end{proof}

We end the present section with a flatness criterion (see Remark \ref{remak:criterion_flatness_holonomy} below). We feel that it might be of independent interest. 

\begin{prop}\label{proposition:criterion_flatness_holonomy}
Let $X$ be a normal complex projective variety of dimension $n$, and let $H$ be an ample Cartier divisor.
Suppose that $X$ is smooth in codimension two. Let $T_X = \sE \oplus \sG$ be a decomposition 
into involutive subsheaves, where $\sE$ is $H$-stable and $\det(\sE)\cong\sO_X$.   
Suppose furthermore that $h^0\big(X,S^{2}(\sE^*)\boxtimes\sL\big)\neq 0$ for some
rank one reflexive sheaf $\sL$ with $c_1(\sL)\cdot H^{n-1}=0$.
Then $c_1(\sE)^2\cdot H^{n-2} = c_2(\sE)\cdot H^{n-2} = 0$.
\end{prop}

\begin{proof}Set $r:=\rank\,\sE$.
Consider a non-zero section $\mathfrak{g} \in H^0\big(X,S^{2}(\sE^*)\boxtimes\sL\big)$. Since $\sE$ is 
$H$-stable, $\mathfrak{g}$ induces an isomorphism
$\sE \cong \sE^*\boxtimes\sL$. Taking determinants and double duals, we obtain
$\sL^{[\otimes r]}\cong \sO_{X}$. 
Let $f\colon \wt X \to X$ be the corresponding cyclic cover (see for instance \cite[Lemma 2.53]{kollar_mori}).
Then $f^{[*]}\sL\cong \sO_{\wt X}$, 
and $\mathfrak{g}$ induces a holomorphic metric on ${f^{[*]}(\sE^*)}_{|{\wt X}_\textup{reg}}$.
Applying Proposition \ref{proposition:holomorphic_connection}, we see that 
${f^{[*]}\sE}_{|{\wt X}_\textup{reg}}$ admits a holomorphic connection. Notice that 
$\wt X \setminus {\wt X}_\textup{reg}$ has codimension at least three. It follows
from Remark \ref{remark:atiyah} that 
$c_1(\sE)^2\cdot H^{n-2}=\frac{1}{\deg(f)}c_1(f^{[*]}\sE)^2\cdot (f^*H)^{n-2}=0$
and that $c_2(\sE)\cdot H^{n-2}=\frac{1}{\deg(f)}c_2(f^{[*]}\sE)\cdot (f^*H)^{n-2}=0$, proving the proposition.
\end{proof}

\begin{rem}\label{remak:criterion_flatness_holonomy}
In the setup of Proposition \ref{proposition:criterion_flatness_holonomy}, suppose furthermore that $X$ is a $\mathbb{Q}$-factorial variety with only canonical singularities.
Then it follows from \cite[Theorem 6.5]{gkp_movable} that there exists a finite surjective morphism
$f\colon \wt X \to X$, \'etale in codimension one, such that $f^{[*]}\sE$ is a locally free, flat sheaf on $\wt X$. 
\end{rem}

\section{Proof of Theorem \ref{thm:main_intro}}

We are now in position to prove our main result.

\begin{prop}\label{prop:main}
Let $X$ be a normal complex projective variety of dimension at most $5$,
with at worst klt singularities. Assume that $K_X \equiv 0$. Then there exists an abelian variety $A$
as well as a projective variety $\wt X$ with at worst canonical
singularities, a finite cover $f: A \times \wt X \to X$, \'etale in
codimension one, and a decomposition
$$\wt X \cong \prod_{j\in J} Y_j$$
such that the induced decomposition of $T_{\wt X}$
agree with the decomposition given by Theorem \ref{thm:infinitesimal_beauville_bogomolov}.
\end{prop}

\begin{proof}Notice first that $K_X \sim_\mathbb{Q} 0$ by \cite[Corollary V 4.9]{nakayama04}. Thus, there exists a 
finite cover $f_1 \colon X_1 \to X$, \'etale in codimension one, such that $K_{X_1}\sim_\mathbb{Z} 0$ (see 
\cite[Lemma 2.53]{kollar_mori}).
By \cite[Proposition 3.16]{kollar97}, $X_1$ has at worst klt singularities.
It follows that $X_1$ has canonical singularities since $K_{X_1}$ is a Cartier divisor. Applying Theorem \ref{thm:infinitesimal_beauville_bogomolov}, we conclude that
there exists an abelian variety $A$
as well as a projective variety $X_2$ with at worst canonical
singularities, a finite cover $f_2: A \times X_2 \to X_1$, \'etale in
codimension one, and a decomposition
$$
T_{X_2} = \oplus_{i\in I} \sE_i
$$
such that the following holds.
\begin{enumerate}
\item The $\sE_i$ are integrable subsheaves of $T_{X_2}$, with $\det(\sE_i)\cong\sO_{X_2}$.
\item The sheaves $\sE_i$ are strongly stable.
\item The augemented irregularity of $X_2$ is zero.
\end{enumerate}

To prove the theorem, we will show that the foliations $\sE_i$ are algebraically integrable.
Let $\beta \colon \wh X_2 \to X_2$ be a $\mathbb{Q}$-factorial terminalization of $X_2$.
By Lemma \ref{lemma:reduction_terminalization},
 there is a decomposition $$T_{\wh X_2} = \oplus_{i\in I} \wh\sE_i$$
of $T_{\wh X_2}$
into involutive subsheaves with $\det(\wh \sE_i)\cong \sO_{\wh X_2}$ such that 
$\sE_i \cong (\beta_*\wh\sE_i)^{**}$. Notice that $\wt q(\wh X_2)=0$ by Lemma \ref{lemma:augmented_irregularity_birational_morphism}. Let $$T_{\wh X_2} = \oplus_{i\in J} \wh\sG_j$$ be a decomposition of $T_{\wh X_2}$
into strongly stable involutive subsheaves with $\det(\sG_i)\cong \sO_{\wh X_2}$ 
whose existence is guarantee by Theorem \ref{thm:infinitesimal_beauville_bogomolov}. 
For any $j\in J$, $\wh \sG_j$ is a direct summand of  
$\sE_{i_j}$ for some $i_j\in I$. To prove the claim,
it suffices to prove that the foliations $\wh \sG_i$ are algebraically integrable.
By Corollary \ref{corollary:augmented_irregularity_versus_second_chern_class}, we must have 
$c_2(\wh\sG_j)\not\equiv 0$ for each $j\in J$. This implies in particular that $\wh\sG_j$ has rank at least $2$, and therefore, $\rank \, \wh\sG_j \in\{2,3\}$ since $\dim X \le 5$ by assumption.
Now, by Theorem \ref{theorem:pereira_touzet_conjecture}, we conclude that the sheaves $\sE_i$ are algebraically integrable, proving our claim.
The proposition then follows from Proposition \ref{proposition:alg_int_towards_dec}.
\end{proof}

\begin{proof}[Proof of Theorem \ref{thm:main_intro}]
Theorem \ref{thm:main_intro} is an immediate consequence of Proposition \ref{prop:main}
and of the characterisation \cite[Proposition 8.21]{gkp_bo_bo} of canonical varieties with
trivial canonical class and strongly stable tangent bundle as singular analogues
of Calabi-Yau or irreducible holomorphic symplectic manifolds.
\end{proof}


\newcommand{\etalchar}[1]{$^{#1}$}
\providecommand{\bysame}{\leavevmode\hbox to3em{\hrulefill}\thinspace}
\providecommand{\MR}{\relax\ifhmode\unskip\space\fi MR }
\providecommand{\MRhref}[2]{%
  \href{http://www.ams.org/mathscinet-getitem?mr=#1}{#2}
}
\providecommand{\href}[2]{#2}

\end{document}